\documentclass[12pt,letterpaper]{amsart}
\usepackage{tikz}
\usepackage{comment}
\usepackage{ytableau}
\usepackage{young} 
\usepackage{hyperref}
\hypersetup{
    colorlinks=true,
    allcolors= {violet},
    menucolor={black},
}

\theoremstyle{plain}
\newtheorem{theorem}{Theorem}
\newtheorem{lemma}[theorem]{Lemma}
\newtheorem{conjecture}[theorem]{Conjecture}
\newtheorem{corollary}[theorem]{Corollary}

\newtheorem{proposition}[theorem]{Proposition}
\theoremstyle{definition}
\newtheorem{definition}[theorem]{Definition}
\newtheorem{remark}[theorem]{Remark}
\newtheorem{example}[theorem]{Example}
\newtheorem{question}[theorem]{Question}

\usepackage{enumitem} % to specify [itemsep=?] option

\newcommand{\old}[1]{}
\title{ On a positivity conjecture in the character table of $S_n$}

\author{Sheila Sundaram}
\address{Pierrepont School, One Sylvan Road North, Westport, CT 06880}
\email{shsund@comcast.net}
\date{20 December 2018}

\begin{document}
\begin{abstract} 
In previous work of this author it was conjectured that the sum of power sums $p_\lambda,$ for partitions $\lambda$ ranging over an interval $[(1^n), \mu]$ in reverse lexicographic order, is Schur-positive.  
Here we investigate this conjecture and establish its truth in the following special cases:
  for $\mu\in [(n-4,1^4), (n)]$  or $\mu\in [(1^n), (3,1^{n-3})], $ or $\mu=(3, 2^k, 1^r)$ when $k\geq 1$ and $0\leq r\leq 2.$  Many new Schur positivity questions are presented.

\end{abstract}
\maketitle
%\begin{center} DRAFT \end{center}

%\vfill\eject

\section{Introduction and Preliminaries}
 
 In this paper we consider Schur positivity questions related to  
  the reverse lexicographic order on integer partitions.  Recall that this total order is defined as follows \cite[p. 6]{M}.  For partitions $\lambda, \mu$ of the same integer $n,$ we say a partition $\lambda$ is preceded by a partition $\mu$ in reverse lexicographic order  if $\lambda_1>\mu_1$ or there is an index $j\geq 2$ such that $\lambda_i=\mu_i$ for $i<j$ and $\lambda_j>\mu_j.$  Thus for $n=4$ we have the total order 
$(1^4) < (2,1^2)< (2^2) <(3,1) <(4).$ In particular our convention  is that the minimal and maximal elements in this total order are $(1^n)$ and $(n)$   respectively.  Our primary goal is to  address the following conjecture:

\begin{conjecture}\label{conj1}\cite[Conjecture 1]{Su1} Let $L_n$ denote the reverse lexicographic ordering on the set of partitions of $n.$ Then the sum of power sum symmetric functions $\sum p_\lambda,$ taken over any initial segment of the total order $L_n,$
i.e. any interval of the form $[ (1^n), \mu]$ for fixed $\mu,$  (and thus  necessarily including  the partition $(1^n)$), is Schur-positive. 
\end{conjecture}

In general, for arbitrary subsets $T$ of partitions of $n$ with $(1^n)\in T,$ the sums $\sum_{\mu\in T} p_\mu$ define (possibly virtual) representations of the symmetric group $S_n,$ of dimension $n!$
There are many instances where Schur positivity fails;  
see the remarks following Example~\ref{ex1.4}. Proposition~\ref{prop4.1} in Section 4 gives  a lower bound for the number of failures.

Conjecture 1 has an equivalent formulation in terms of the character table of  $S_n.$ If the columns of the table are indexed by the integer partitions of $n$ corresponding to the conjugacy classes, in reverse lexicographic order, left to right, and the rows by the irreducible characters (hence also corresponding to partitions), then the conjecture states that, for each row, indexed by some fixed partition $\lambda$ of $n$, the sum of the entries in the first $k$ consecutive columns, beginning with the column indexed by $(1^n)$, is a nonnegative integer.  If the $k$th column corresponds to the conjugacy class indexed  by the partition $\mu,$  this row sum is the multiplicity of the Schur function $s_\lambda$ in the sum $\sum_{\nu\in[(1^n), \mu]} p_\nu.$

The questions treated in this paper are intimately connected to (and indeed motivated by) the representation theory of the symmetric group and the general linear group.  For background on these topics we refer the reader to \cite{M} and \cite{St4EC2}, in particular the Exercises in the latter reference.

Let $\psi_n$ denote the Frobenius characteristic of the conjugacy action on $S_n.$  The orbits of this action are the conjugacy classes. Let $f_n$ denote the Frobenius characteristic of the conjugacy action of $S_n$ on the class of $n$-cycles.
%obtained by inducing the trivial representation of the cyclic group of order $n$ generated by an $n$-cycle,  up to $S_n.$ 
Let $h_n, e_n$ denote respectively the homogeneous and elementary symmetric functions of degree $n,$ and let $[\quad  ]$ denote plethysm. By a general observation  of Solomon \cite{So} for finite groups (see also \cite[Exercise 7.71]{St4EC2}, \cite[Corollary 4.3]{Su1}), we have the following facts.  
In view of Part (2) of the theorem below, Conjecture 1 may be seen as a generalisation of the Schur positivity of  the sum of all power sums $\sum_{\lambda\vdash n} p_\lambda.$

\begin{theorem}\label{thm1.1} The Frobenius characteristic $\psi_n$ of the conjugacy action of $S_n$ admits the following decompositions:
\begin{enumerate}[label=(\arabic*)]
\item[(1)] 
$\psi_n=\sum_{\lambda\vdash n} \prod_i h_{m_i}[f_i],$  where the partition $\lambda$ has $m_i$ parts equal to $i.$ 
%(Orbit decomposition)
\item[(2)] $\psi_n=\sum_{\lambda\vdash n} p_\lambda,$ and hence the latter sum is Schur-positive.
%( rows of character table)
\end{enumerate}
\end{theorem}
 
 %That is, $\psi_n$ is the degree $n$ term in $\prod_{j\geq 1} (1-p_j)^{-1}.$  

\begin{definition}\label{def1.2} If $\mu$ is a partition of $n,$ we write $\psi_\mu$ for the sum of power sums 
\begin{center}$\sum_{\lambda\in[(1^n), \mu]} p_\lambda.$\end{center}

More generally if $T$ is any subset of partitions of $n,$ define $\psi_T$ to be the sum $\sum_{\mu\in T} p_\mu.$
\end{definition}
Thus $\psi_{(n)}=\psi_n,$ and  the multiplicity of the Schur function $s_\lambda$ in $\psi_\mu$ is the sum of the values of the irreducible character $\chi^\lambda$ on the conjugacy classes in the interval $[(1^n), \mu].$  

Clearly $\psi_{(1^n)}$ is just the characteristic of the regular representation.  Also since $\psi_2=2 s_{(2)}$ is twice the trivial representation, $\psi_{(2,1^{n-2})}=p_1^{n-2}\psi_{(2)}=2 s_{(2)} p_1^{n-2}.$  %Coincidentally this is also the representation on the total Whitney homology of $\Pi_n$, or the cohomology ring of the complement of the braid arrangement for type $A_{n-1}$ \cite[Corollary 5.3]{Su1}.
The Schur function expansion of  $\psi_n$ for $n\leq 10 $ appears in \cite[Table 1]{Su1}.   We have verified Conjecture 1 in Maple up to $n=20.$ %See Tables 4 and 5  at the end of this paper for  the $\psi_\mu,$  $n\leq 6.$

\vskip .1in

The main result of this paper gives an affirmative answer to Conjecture 1 in the following cases:

\begin{theorem}\label{thm1.3}  The symmetric function $\psi_\mu=\sum_{(1^n)\leq \lambda\leq \mu} p_\lambda$ is Schur-positive if  $\mu\leq (3, 1^{n-3})$ or $\mu\geq (n-4,1^4)$ in reverse lexicographic order.
\end{theorem}

%2018 March 11 :couldn't prove this :
%\begin{theorem}  $\psi_\mu$ is Schur-positive if $\mu\geq (n-3,3)$ or $\mu\leq (3, 2^k, 1^{n-3-2k})$ in reverse lexicographic order, where $k=\lfloor\frac{n-3}{2}\rfloor.$
%\end{theorem}

Our approach to Conjecture 1 proceeds in two directions.
One can start at the bottom of the chain,  with $p_1^n$ (which  contains all irreducibles), and add successive $p_\lambda$'s going up  the chain.  The arguments in this case are subtle, and give an interesting decomposition of the corresponding representation.  See Theorem~\ref{thm2.11}. Alternatively, one can start at the top of the chain, with the known Schur-positive function $\psi_n,$ which is also known to contain all irreducibles (see Section 2), and examine what happens to the irreducibles upon subtracting  successive $p_\lambda$'s going down  the chain, from $\psi_n.$ This is done in Theorem~\ref{thm2.16}, and requires a careful analysis (Lemmas~\ref{lem2.12} to~\ref{lem2.14}) of the Schur functions appearing in products of power sums. 
The technical difficulty here is in ensuring that the resulting expressions (Proposition~\ref{prop2.15}) are \textit{reduced}, i.e. each term corresponds to a unique Schur function.  
 The argument now hinges on the following fact:  no irreducible in the partial sum of power sums appears with multiplicity exceeding the   lower bound, established in Lemma~\ref{lem2.6}, for the multiplicity of each irreducible in $\psi_n$.  

The proof of Theorem~\ref{thm2.11} hints at interesting properties of the representations $\psi_{(2^k)}.$ In Section 3 we present  conjectures suggested by that proof, and establish more Schur positivity results (the case $\mu=(3,2^k,1^r)$ for $0\leq r\leq 2$, Proposition~\ref{prop3.7}), as well as  generalisations of Theorem~\ref{thm2.16} to the twisted conjugacy action as defined in \cite{Su1}.   Section 4 concludes the paper with an analysis of the number of subsets of partitions whose associated sum of power sums is \textit{not} Schur-positive.  Tables of the Schur expansion of $\psi_\mu$ appear in Section 5.

An interesting implication  of Conjecture 1 is obtained by  examining the occurrence of the sign representation in $\psi_\mu.$ 
The multiplicity of $s_{(1^n)}$ in $\psi_\mu$  is clearly
\begin{equation}\label{1.1}\sum_{(1^n)\leq \lambda\leq \mu} (-1)^{n-\ell(\lambda)}, \end{equation}
where $\ell(\lambda)$ denotes the number of parts of $\lambda,$ 
because the value of the sign character on the conjugacy class indexed by 
 $\lambda$ is $(-1)^{n-\ell(\lambda)}.$
 In our reverse lexicographic ordering, these values are the partial sums (computed left to right, with the left-most column indexing the class of the identity $(1^n)$ ) in the first row of the character table.
Hence Conjecture 1 implies that the sum~\eqref{1.1} is nonnegative for all $\mu\vdash n.$ 
When $\psi_\mu=\psi_n$ is the characteristic of the full conjugacy action of $S_n,$   the expression~\eqref{1.1}  can be shown to equal 
%%(e.g. by a routine generating function argument)  
 the number of partitions of $n$ all of whose parts are odd and distinct, or equivalently, the number of self-conjugate partitions of $n$  \cite[Proposition 4.21]{Su1}.  In general the character values form a sequence of 1's and $(-1)$'s, with the partitions  written  in reverse lexicographic order; it is not  obvious why the resulting partial sums should be nonnegative.   The nonnegativity is easily  checked for $n\leq 7.$ %because those character tables are small.  
 The  example below contains  data for $8\leq n\leq 13.$  As observed above, the last partial sum in each case is the number of partitions with all parts odd and distinct.  Theorem~\ref{thm2.16} will imply (see Corollary~\ref{cor2.17}) that at least the last five partial sums are necessarily  positive.

\begin{example}\label{ex1.4}  (See the discussion following Corollary~\ref{cor2.17} for the meaning of the italics and the underlined runs.)
The  22 values of the sign character for $S_8,$ on the conjugacy classes taken in reverse lexicographic order, beginning with 
$(1^n),$ are 

${\bf 1, -1, 1, -1, 1, 1, -1, 1, 1, -1, -1, 1, -1, -1, 1, 1, -1, 1, -1, 1, 1, -1},$ 
$$\text{\small with  partial sums: }{\bf 1, 0, 1, \underline{0, 1, 2}, \underline{1, 2, {\it 3},} {\it 2, 1}, {\it 2, 1}, \underline{{\it 0}, 1, 2}, 1, 2, \underline{1, 2, 3}, 2};$$
the 30 values for $S_9$ are 

${\scriptstyle \bf 1, -1, 1, -1, 1, 1, -1, 1, -1, 1, -1, 1, -1, 1, -1, -1,  \underline{1, 1, 1}, -1, 1, 1, -1, -1, 1, -1, 1, -1, -1, 1}, $

\noindent  with partial sums:
$${ \bf 1, 0, 1, \underline{0, 1, 2}, 1, 2, 1, 2, 1, 2, 1, {\it 2, 1},\, \underline{{\it 0}, 1, 2, 3}, 2, 3, 4, 3, 2, 3, 2, 3, 2, 1, 2
};$$
and the 42 values for $S_{10}$ are 
$${\scriptstyle \bf 
1, -1, 1, -1, 1, -1, 1, -1, 1, -1, 1, -1, 1, 1, -1, 1, -1, 1, -1, 1, -1, 1, -1, 1, -1, \underline{1, 1}, {\it -1, -1}, 1, -1, 1, -1, -1, \underline{1, 1}, -1, 1, -1, \underline{1, 1}, -1
}, $$
with sequence of partial sums:

\noindent
\textbf{\tiny
1, 0, 1, 0, 1, 0, 1, 0, 1, 0, 1, \underline{0, 1, 2}, 1, 2, 1,
   2, 1, 2, 1, 2, 1, 2, \underline{1, 2, {\it 3}}, {\it 2, 1}, 2, 1, 2, 1, 0, 1, 2, 1, 2, \underline{1,  2, 3}, 2.}
  
  The partial sums for $S_{11}$ are:
  
  \noindent
 \textbf{\tiny 1, 0, 1, 0, 1, 0, 1, 0, 1, \underline{0, 1, 2}, \underline{1, 2, {\it 3}}, {\it 2, 1},
   2, 1, 2, 1, {\it 2, 1, 0}, 1, \underline{0, 1, 2}, 1, 2, 1, 2, 1, 2, \underline{1, 2, 3}, 2, 3,
 2, \underline{1, 2, 3}, 2, 3, \underline{2, 3, {\it 4}}, {\it 3, 2}, 3, 2, {\it 3, 2, 1}, 2}
 
 The partial sums for $S_{12}$ are:
  
  \noindent
 \textbf{\tiny 1, 0, 1, 0, 1, \underline{0, 1, 2}, 1, 2, \underline{1, 2, 3}, 
 2, 3, 2, 3,
    2, 3, 2, 3, 2, {\it 3, 2, 1}, {\it 2, 1, 0}, 1, 2, \underline{1, 2, 3}, 2, 3, 2, 3, 2, 3,
    \underline{2, 3, 4}, 3, 4, 3, {\it 4, 3, 2}, 3, 2, 3, 2, 3, 2, {\it 3, 2, 1}, 2, 3, \underline{2, 3, {\it 4}}, {\it 3, 2}, 3, 2, 
{\it 3, 2}, \underline{{\it 1}, 2, 3}, 2, 3, \underline{2, 3, 4}, 3 }
 
 The partial sums for $S_{13}$ are:
  
  \noindent
 \textbf{\tiny 1, 0, 1, 0, 1, \underline{0, 1, 2}, 1, 2, 1, 2, 1, 2, 1, 2, 1,
    2, \underline{1, 2, 3}, 2, 3, 2, {\it 3, 2, 1}, 2, 1, 2, 1, 2, 1, 2, \underline{1, 2, {\it 3}}, {\it 2, 1},
    2, 1, 2, \underline{1, 2, 3}, \underline{2, 3, {\it 4}}, {\it 3, 2}, {\it 3, 2}, \underline{{\it 1}, 2, 3}, 2, 3, 2, 3, 2, 3,
    2, {\it 3, 2, 1}, 2, 1, 2, \underline{1, 2, 3, 4}, 3, 4, 3, 4, 3, 4, \underline{3, 4, {\it 5}}, {\it 4, 3},
    4, 3, \underline{2, 3, 4}, 3, 4, \underline{3, 4, {\it 5}}, {\it 4, 3}, 4, 3, {\it 4, 3, 2}, 3
 }

\end{example}

In addition to the sign, it is also interesting to examine the multiplicity of the representation $(2, 1^{n-2})$. 
\begin{example}\label{ex1.5} 
The values of the irreducible character indexed by $(2,1^{n-2})$ on the conjugacy classes in reverse lexicographic order are as follows.

For $n=7:$ 
$${\bf 6,-4,2,0,3,-1,-1,0,-2,0,1,1,1,0,-1,}$$
with partial sums: ${\bf 6,2,4,4,7,6,5,5,3,3,4,5,6,6,5}.$

For $n=8:$ 
$${\bf 7, -5, 3,-1,-1,4,-2,0,1,1,-3,1,1,0,-1,2,0,-1,-1,-1, 0,1 }$$
with partial sums: 
${\bf 7,2,5,4,3,7,5,5,6,7,4,5,6,6,5,7,7,6,5,4,4,5}.$  

These examples also highlight the fact that there are many ways of reordering the conjugacy classes so that the resulting partial (row) sums in the character table may be negative, and the corresponding sum of power sums will thus fail to be Schur-positive.  We will return to this observation in Section 4.
\end{example}
%
%\begin{figure}
\begin{center}
\begin{tikzpicture}[scale=.7]
  \node (max) at (0,8) {$\scriptstyle(6)$};
  \node (a) at (0,7) {$\scriptstyle(5,1)$};
  \node (b) at (0,6) {$\scriptstyle(4,2)$};
  \node (c1) at (-1,5) {$\scriptstyle(3^2)$};
  \node (c2) at (1,5) {$\scriptstyle(4,1^2)$};
  \node (d) at (0,4) {$\scriptstyle(3,2,1)$};
  \node (e1) at (-1,3) {$\scriptstyle(2^3)$};
  \node (e2) at (1,3) {$\scriptstyle(3,1^2)$};
  \node (f) at (0,2) {$\scriptstyle(2^2,1^2)$};
  \node (g) at (0,1) {$\scriptstyle(2,1^4)$};
  \node (min) at (0,0) {$\scriptstyle(1^6)$};
  \draw (min) -- (g) -- (f) -- (e1) -- (d) -- (c1)-- (b)-- (a) -- (max) ;
   \draw (min) -- (g) -- (f) -- (e2) -- (d) -- (c2)-- (b)-- (a) -- (max) ;
 % \draw[preaction={draw=white, -,line width=6pt}] (a) -- (e) -- (c);
\end{tikzpicture}

{\small Figure 1: Dominance order for partitions of 6}
\end{center}
%\caption{\small Dominance order for partitions of 6}
%\end{figure}

We remark that Conjecture 1 is false if one considers dominance order  instead of reverse lexicographic order.  It fails for the  first case  in which dominance  departs from reverse lexicographic order, $n=6,$ as the following example shows:

\begin{example}\label{ex1.6}
The seven partitions (weakly) dominated by $(4,1^2)$ are (see Figure 1 above)
$$\{(4,1^2), (3,2,1), (2^3), (3,1^3), (2^2,1^2) , (2, 1^4), (1^6)\}$$
\end{example}

  The sum of power sums is thus 
  $p_4p_1^2 +p_3p_2p_1+ p_2^3+p_3 p_1^3+ p_2^2 p_1^2 +p_2 p_1^4 +p_1^6;$
  in the corresponding $S_n$-module,  the sign appears with negative multiplicity  (all other irreducibles occur with positive coefficient): 
  \vskip  -.25in
  %\[{ 7 s_{(6)}\!\! +\! 11 s_{(5, 1)}\! +\! 15 s_{(4, 2)}\! +\! 8 %s_{(4, 1^2)}\! +\! 3 s_{(3^2)}\!  +\! 14 s_{(3, 2, 1)}
%\!  +\! 10 s_{(3, 1^3)}\! +\! 7 s_{(2^3)}\!  +\! 5 s_{(2^2, 1^2)}
%\!+\! 5 s_{(2, 1^4)}
%\! - \!s_{(1^6)}}\]
 \begin{align*}& 7 s_{(6)} + 11 s_{(5, 1)} + 15 s_{(4, 2)} + 8  s_{(4, 1^2)} + 3 s_{(3^2)}  + 14 s_{(3, 2, 1)}\\
 & + 10 s_{(3, 1^3)} + 7 s_{(2^3)}  + 5 s_{(2^2, 1^2)}
+ 5 s_{(2, 1^4)}
 - s_{(1^6)}.\end{align*}
This is the only instance that fails for $n=6.$    For $\mu=(4, 1^{n-4}),$  (the hook with one part equal to 4), similarly, up to $n=12,$  the only irreducible with
negative coefficient is the sign, appearing with coefficient $(-1).$

\section{Intervals in Reverse Lexicographic Order}

The following fact about the representation $\psi_n$  was first proved by Avital Frumkin. See \cite[Solution to Exercise 7.71]{St4EC2} for more references.   

\begin{theorem}\label{thm2.1}\cite{Fr}  If $n\neq 2,$ the representation $\psi_n$ contains all irreducibles.
\end{theorem}

We will need the following stronger result of \cite{Su2} characterising the conjugacy classes containing all irreducibles.  Recall (see \cite{Su2} for  references to the literature) that a conjugacy class in a finite group $G$ is called {\it global} if the orbit of the conjugacy action corresponding to that class contains all irreducibles of $G.$ 

\begin{theorem}\label{thm2.2}\cite[Theorem 5.1]{Su2} Let $n\neq 4, 8.$ Then the conjugacy class indexed by a partition $\lambda$ contains all irreducibles, i.e. it is a global class, if and  only if $\lambda$ has at least two parts, and all its parts are distinct and odd. If $n=8,$ the conjugacy class indexed by $(7,1)$ is global, while the class of the partition $(5,3)$ contains all irreducibles except those indexed by $(4^2)$ and $(2^4).$ 
\end{theorem}

We also require some information on the irreducibles appearing in $f_n,$ the $S_n$-action by conjugation on the class of $n$-cycles.  Since this is a permutation representation with one orbit,  the trivial representation  appears exactly once. It is also easy to see that the sign representation appears only if $n$ is odd.  We will  make use of the following definitive result of Joshua Swanson:

\begin{theorem}\label{thm2.3}(\cite{Sw}, \cite[Lemma 4.2]{Su1}) Let $n\geq 1.$ If $n$ is odd, the representation $f_n$ contains all irreducibles  except those indexed by $(n-1,1)$ and $(2, 1^{n-2}).$  If $n$ is even, $f_n$ contains all irreducibles except $(n-1,1)$ and $(1^n).$ 
\end{theorem}

The result below was stated without proof in \cite{Su1};  we sketch a proof here.  (There is a misprint in the statement of Part (4) in \cite[Proposition 4.21]{Su1} which is  corrected below.) 
\begin{proposition}\label{prop2.4}\cite[Proposition 4.21]{Su1}  The multiplicity in $\psi_n $ of the irreducible indexed by the partition 
\begin{enumerate}[label=(\arabic*)]
\item $(n)$ is $p(n)$, the number of partitions of $n.$
\item $(1^n)$ is the number of partitions of $n$ into parts that are distinct and odd, which is also the number of self-conjugate partitions of $n$.  This multiplicity is nonzero for $n\neq 2.$
\item $(n-1,1)$ is $\sum_{\lambda\vdash n}\left( |\{i: m_i(\lambda)\geq 1\}| -1\right),$ which in turn equals the number of distinct parts in all the partitions of $n,$ minus the number of partitions of $n.$
In particular  this multiplicity is at least the number of non-rectangular partitions of $n,$ and hence at least $(n-1).$
\item  $(2,1^{n-2})$ is 
$\sum_{\lambda\in T_1} (\ell(\lambda)-1) +|T_2\cup T_3|,$
%+\sum_{\lambda\in T_2}|\{i: m_i(\lambda)\geq 1\}|$, 

\noindent
where the first sum runs over  the set $T_1$ of all partitions of $n$ $\lambda$ with parts that are distinct and odd, and  the set $T_2$ consists of partitions $\lambda$ of $n$ with all parts odd and distinct except for  one part of multiplicity 2, while $T_3$ consists of partitions with all parts odd and distinct except for exactly one even part. 
%$m_i(\lambda)\leq 2, m_j(\lambda)=2 \text{ for exactly one part } j.$
\end{enumerate}
\end{proposition}

\begin{proof} Part (1) is clear since $p(n)$ is the number of conjugacy classes of $n.$  As alluded to in the Introduction, Part (2) is a computation of the sum $\sum_{\mu\vdash n} (-1)^{n-\ell(\mu)}$, 
and follows from the standard generating function identity for integer partitions by number of parts.  Since the sign representation always occurs in $f_n$ if $n$ is odd, it also occurs in $f_1 f_{n-1}$ if $n$ is even, i.e. in the conjugacy class $(n-1,1)$ for $n\neq 2.$ The second statement of Part (2) follows.

For Parts (3) and (4), by Frobenius reciprocity, we compute the multiplicity of the trivial (respectively sign) representation in 
the restriction of $\psi_n$ to $S_{n-1},$ using the partial derivative with respect to $p_1$ (see e.g. \cite{M}) and Theorem~\ref{thm1.1}~(1).  Here we also need the fact that $f_n$ always contains exactly one copy of the trivial representation, and contains exactly one copy of the sign representation when $n$ is odd, and none if $n$ is even.  The restriction of $\psi_n$ to $S_{n-1}$ has Frobenius characteristic
\[ \sum_{\lambda\vdash n} \sum_{i:m_i\geq 1} (h_{m_i-1}[f_i]\cdot p_1^{i-1} \cdot \prod_{j\neq i} h_{m_j}[f_j]). \]

Let $\lambda\vdash n.$ For each $i$ such that $m_i(\lambda)\ge 1,$ the inner sum contributes 1 to the multiplicity of the trivial representation  for each distinct part $i$ of $\lambda.$
Thus  in Part (3), since  $s_{(n-1,1)}=h_{n-1}h_1-h_n,$  the required multiplicity is $p(n)$ less than the multiplicity of $s_{(n-1)}$ in the restriction of $\psi_n$ to $S_{n-1},$ and the result follows.

Part (4) follows by a similar analysis, using standard facts about how the sign representation restricts to wreath products.  
% Here we need to use the fact that $$sgn\downarrow_{S_m[S_n]} = \omega^n(h_m)[e_n].$$ 
Again let $\lambda\vdash n,$ and consider a fixed $i$ such that $m_i\ge 1.$ 
By restricting the sign to the appropriate subgroup, we see that the multiplicity in $h_{m_j}[f_j]$ is nonzero (and thus equal to 1) if and only if $j$ is odd and $m_j=1.$  Similarly the 
multiplicity in $h_{m_i-1}[f_i]$ is nonzero (and thus equal to 1) if and only if either $i$ is odd and $m_i-1=1,$ or $i$ is even and $m_i-1=0.$ 

This implies that every term $h_{m_i-1}[f_i]\cdot p_1^{i-1} \cdot \prod_{j\neq i} h_{m_j}[f_j]$ contributes 1 to the multiplicity if $\lambda$ has all parts odd and distinct, but if $\lambda\in T_2\cup T_3,$ only one of these terms,  
namely when $i$ is either the odd part of multiplicity 2, or the unique even part, makes a nonzero contribution (equal to 1).  For all other $\lambda$ the multiplicity is zero.  This completes the proof. 
\end{proof}

%Hence for each fixed $\lambda\vdash n,$ the multiplicity of the sign representation is 
%
%$\begin{cases} \ell(\lambda), & \text{ if } \lambda\in T_1,\\
%1 & \text{ if }\lambda\in T_2\cup T_3 .\end{cases}$
%
%This finishes the proof.
%\end{proof}

\begin{lemma}\label{lem2.5} \begin{enumerate}[label=(\arabic*)]  
\item Let $n\neq 3.$ Then 
$f_n f_1$ contains  all irreducibles except $ (1^{n+1})$ if $n$ is even. 
\item Let $n\geq 5,$ and let $n$ be odd. Then the product $f_n f_2$ contains all irreducibles except the sign.
\item Let $n\geq 5.$ If $n$ is odd,  every irreducible except the sign appears in each of the conjugacy classes $(n-2,2)$ and $(n-2,1,1).$ 
\item Let $n\geq 6.$ If $n$ is even, every irreducible except the sign appears in $f_{n-3}f_2 f_1.$
\item Let $n\geq 6$ be even.  Then every irreducible appears in 
$f_{n} f_2 $ except for the sign and the one indexed by $(2,1^{n}),$ which does however appear in $ f_{n-2} f_3 f_1.$  In particular all the irreducibles except for the sign appear among the two conjugacy classes  
 $(n,2)$ and $(n-2,3,1).$
\end{enumerate}
\end{lemma}
\begin{proof}  For {\bf Part (1)}: The result is clear for $n=1,2$ so assume $n\geq 4.$ 

If $n$ is odd this is immediate from Theorem~\ref{thm2.2}. If $n$ is even, then by Theorem~\ref{thm2.3}, $f_n$ contains all irreducibles except $(1^n)$ and $(n-1,1).$ Now $s_{(n-1,1)}\cdot f_1=s_{(n+1)}+s_{(n,1)}+s_{(n-1,2)}. $
But each of these summands appears in the product $g_n\cdot f_1$ for 
$g_n=s_{(n)}, s_{(n)}, s_{(n-2,2)} $ respectively, and each $g_n$ appears in $f_n.$  The only irreducible that does not appear is $(1^{n+1}).$ 

For {\bf Part (2)}: It is easy to compute, since $f_2=h_2$ and $f_3=h_3+e_3,$ 
$f_3 f_2= s_{(5)}+s_{(4,1)}+s_{(3,2)}+s_{(3,1,1)}+s_{(2,1,1,1)},$ so the product does not contain $s_{(2,2,1)}.$ 

So let $n\geq 5$  be odd.  By Theorem~\ref{thm2.3},  $f_n$ contains all irreducibles except those indexed by $(n-1,1)$ and $(2, 1^{n-2}).$ 
  We have 
$$s_{(n-1,1)}\cdot f_2=s_{(n-1,1)}\cdot h_2=s_{(n+1,1)}+s_{(n,2)}+s_{(n,1,1)}+s_{(n-1,3)}+s_{(n-1,2,1)}.\quad (A) $$
The first two summands appear in $g_n\cdot f_2$ for $g_n=s_{(n)}\cdot f_2.$ The last two summands appear in  $g_n\cdot f_2$ for $g_n=s_{(n-1,2)}.$ Finally $s_{(n-1,1,1)}$ appears in the product 
$g_n\cdot f_2$ for $g_n=s_{(n-2,1,1)},$ and this appears in $f_n$ for 
$n\geq 5.$  Thus in all cases $g_n$ appears in $f_n.$ 

Next consider the other missing irreducible, $(2,1^{n-2})$. We have 
$$s_{(2,1^{n-2})}\cdot f_2=s_{(2,1^{n-2})}\cdot h_2=s_{(4,1^{n-2})}+s_{(3,2,1^{n-3})}+s_{(3,1^{n-1})}+s_{(2^2, 1^{n-2})}. \quad (B)$$
The first three appear in the product $g'_n\cdot f_2$ for $g'_n=s_{(3, 1^{n-3})},$ which is a constituent of $f_n, n\geq 5.$  The last one appears in the product $g'_n\cdot f_2$ for $g'_n=s_{(2^2, 1^{n-4})},$ 
which again is a constituent of $f_n, n\geq 5.$  This completes the   argument.

For {\bf Part (3)}, observe that the conjugacy classes indexed by $(n,2)$ and $(n,1,1)$ both afford the same representation, namely $f_n \cdot h_2=f_n \cdot f_2.$  The result now follows from Part (2).

{\bf Part (4)} follows by applying Part (2) to $f_{n-3} f_2,$ since $n-3$ is odd.

For {\bf Part (5)}: Since $n$ is even, again Theorem~\ref{thm2.3} tells us that $f_n$ contains all irreducibles except those indexed by $(n-1,1)$ and $(1^{n}).$ 
From (A) above we see that the product $f_{n} f_2$ may miss the irreducibles indexed by 
$$(n+1,1), (n,2), (n,1^2), (n-1,3) \text{ and } (n-1,2,1),$$
but all these appear in the set of products $\{s_{(n)} f_2, s_{(n-2,1^2)}f_2\}$, and $s_{(n)}, s_{(n-2,1^2)}$ appear in $f_n.$ 
The only other irreducibles possibly missed by the product $f_{n} f_2$ are those in the product $s_{(1^n)} f_2,$ namely
 $$ (2, 1^{n-2}), (3, 1^{n-3}) \text{ and }(1^{n+2}).$$
 Clearly $s_{(3, 1^{n-3})}$ occurs in $s_{(2, 1^{n-4})}\cdot f_2,$ and 
 $f_n$ contains $s_{(2, 1^{n-4})}$ since $n$ is even.
 
 To establish the claim, we now need only show that the irreducible 
 $(2, 1^{n-2})$ appears in $f_{n-2}f_3 f_1.$ But  $f_3=h_3+e_3,$ so  $f_{n-2}f_3$ contains all the irreducibles in the product $f_{n-2} e_3.$ 
 Since $n-2$ is even, it contains $s_{(2, 1^{n-4})}$ and this finishes the argument.
\end{proof}

\begin{lemma}\label{lem2.6} Let $n\geq 5.$ Let $do_n$ denote the number of partitions of $n$ with at least two parts and with all parts odd and distinct.  In  the conjugacy representation 
$\psi_n,$ every irreducible except possibly the sign occurs with multiplicity at least 
$\begin{cases} 4+do_n, & n \text{ odd};\\
                      3+do_n, & n \text{ even}.
\end{cases}$
 This number is at least 5 for odd $n\geq 7,$ and at least 4 for even $n\geq 6.$
 \end{lemma}

\begin{proof} First let $n$ be odd.  By Lemma~\ref{lem2.5}, the following conjugacy classes contain all irreducibles except the sign: $(n-1,1), (n-2,2), (n-2,1^2).$ Also by Theorem~\ref{thm2.3}, the conjugacy class $(n)$ (or equivalently the symmetric function $f_n$) contains all irreducibles except for the one indexed by $(n-1,1)$ and $(1^n). $  But the irreducible $(n-1,1)$ appears at least $n-1$ times in $\psi_n,$ by Proposition~\ref{prop2.4}.  Thus we have multiplicity at least 4 for each irreducible. 
Since none of the four conjugacy classes listed above is global by Theorem~\ref{thm2.1}, we have a multiplicity of at least $4$ plus the number of 
global classes.

Now let $n$ be even, $n\geq 8.$  (The case $n=6$ can be checked by direct computation.  See, e.g. \cite[Table 1]{Su1}. Then by Theorem~\ref{thm2.3}, the conjugacy class $(n)$ has all irreducibles except for $(n-1,1)$ and $(1^n).$  Also by Lemma~\ref{lem2.5}, $f_{n-3} f_2$ has all irreducibles except for $(1^n).$  Hence so does the conjugacy class $(n-3,2,1).$ Finally this is also true by Lemma~\ref{lem2.5} again, 
for  the sum 
$(f_{n-2} f_2+f_{n-4}f_3 f_1).$  We have accounted for a multiplicity of at least 3 for every irreducible except the sign, in addition to the global classes.

We now show that the number of global classes is at least $\lfloor{\frac{k}{2}}\rfloor$ if $n=2k, 2k+1$.

First let $n=2k \geq 6$ be even.
In this case, applying Theorem~\ref{thm2.2},  we have at least
 $\lfloor{\frac{k}{2}}\rfloor\geq 1$  global conjugacy classes:
$\{(2k-r,r):r=1, 3, \ldots, 2\lfloor{\frac{k}{2}}\rfloor-1\}.$  
If $n=2k+1\geq 9,$ then again there at least 
$\lfloor{\frac{k}{2}}\rfloor-1\geq 1$  global conjugacy classes in the set 
$\{(2k-r,r,1):r= 3, 5,\ldots, 2\lfloor{\frac{k}{2}}\rfloor-1\}$ and 
one more: $(2k-7,5,3).$   \end{proof}

\begin{remark}\label{rk2.7}Tables of the decomposition into irreducibles for $\psi_n, n\leq 10,$ are given in \cite{Su1}.  We point out a misprint in Table 1 of \cite{Su1} for $n=7:$ the fifth entry from the bottom, for the multiplicity of $(3,1^4)$ 
in $\psi_7,$ should be $13,$ not $7$. From this data, the truth of the lemma follows for $n\leq 10.$   It is worth noting that the tables indicate far greater  lower bounds than we have just established, for the multiplicity of the irreducible indexed by $\mu$ when $\mu\neq (1^n), (2,1^{n-2})$.\end{remark}

%\vfill\eject
%By definition of the Frobenius characteristic map, it follows that the coefficient of the Schur function $s_\lambda$ in the $n$th power-sum $p_n$ is the character value $\chi^\lambda ((n)).$ Hence we have:

We begin our analysis by directing our attention to the bottom of the chain, to  examine the representations $\psi_\mu$ for $\mu >(1^n).$   Our argument in this case is  somewhat mysterious.  One  interesting  aspect is  the role played by the following calculation.
\begin{lemma}\label{lem2.8}   The symmetric function $p_2^2+h_2 e_2=h_2^2-e_2h_2+e_2^2$ is Schur-positive.
\end{lemma}
\begin{proof}  Note that $p_2=h_2-e_2.$ It is straightforward to compute, using  Pieri rules (see e.g. \cite{M}), that $h_2^2-e_2h_2+e_2^2=s_{(4)}+s_{(1^4)}+2s_{(2,2)}.$
\end{proof}
\begin{definition}\label{def2.9} Let $T$ be any subset of integer partitions. Denote by $p_{n,T}$ the sum of power sums $\sum_{\lambda\vdash n, \lambda \in T} p_\lambda.$ 
%$\sum_{\lambda\vdash n:\lambda_i \in T\text{ for all }i} p_\lambda.$ 
\end{definition}
\begin{theorem}\label{thm2.10}\cite[Theorem 4.23]{Su1} If $T=\{\lambda\vdash n: \lambda_i=1,2 \text{ for all }i\}$ then $p_{n,T}$ is Schur-positive.  We have $p_{2m+1,T}=p_1 p_{2m, T}$ and 
\begin{equation*} p_{2m, T}
=\sum_{\stackrel{j=1} {j \text{ odd}}}^{m+1} {{m+1}\choose {j}} h_2^{m+1-j} e_2^{j-1}
\end{equation*}
\end{theorem}
\begin{theorem}\label{thm2.11}  Let $\mu$ be a partition in the interval $[(1^n), (3, 1^{n-3})].$ Then $\psi_\mu$ is Schur-positive.  Equivalently, the following are Schur-positive:
\begin{enumerate}[label=(\arabic*)]
\item $ \psi_{(2^k,1^{n-2k})}=\sum_{i=0}^k p_2^i p_1^{n-2i},$ for $k\leq n/2;$ one has the recurrence $$\psi_{(2^{k+1},1^{n-2(k+1)})}=\psi_{(2^k,1^{n-2k})}+p_2^{k+1} p_1^{n-2(k+1)}, 0\leq k< n/2.$$
\item $\psi_{(3, 1^{n-3})}=p_3p_1^{n-3} +\psi_{(2^k,1^{n-2k})},$ where $k=\lfloor\frac{n}{2}\rfloor.$
\end{enumerate}
\end{theorem}
\begin{proof}
For Part (1): \begin{center} $ \psi_{(2^k,1^{n-2k})}=\sum_{i=0}^k p_2^i p_1^{n-2i}=p_1^{n-2k} \sum_{i=0}^k p_2^i p_1^{2k-2i}
= p_1^{n-2k}p_{2k,T},$\end{center}
 where $T=\{\lambda\vdash 2k: \lambda_i=1,2 \text{ for all }i\}.$ 
But by Theorem~\ref{thm2.10} 
 we know that $p_{2k,T}$ is Schur-positive as a representation of $S_{2k}.$ 

For Part (2): 
Writing $m=\lfloor\frac{n}{2}\rfloor,$  since in reverse lexicographic order, $(3,1^{n-3})$ covers the partition with at most one part equal to 1 and all other parts equal to 2, we have,

$$\psi_{(3, 1^{n-3})}=\psi_{(2^m, 1^{n-2m})}+p_3p_1^{n-3}.$$
Note that $p_3=h_3+e_3 -(h_2 h_1-h_3)=2 s_{(3)} +s_{(1^3)} 
-h_2 p_1.$ 
Hence 
$$\psi_{(3, 1^{n-3})}=\psi_{(2^m, 1^{n-2m})}-h_2p_1^{n-2}
+(2 s_{(3)}+s_{(1^3)}) p_1^{n-3}.$$

We will establish the stronger claim that 
\begin{equation}\label{eq2.1}\psi_{(2^m, 1^{n-2m})}-h_2p_1^{n-2}\end{equation} is Schur-positive.
From Theorem~\ref{thm2.10}, it suffices to assume that $n=2m.$ 
In this case, with $T$ being the set of partitions of $2m$ with parts 1,2,  we have 
\begin{equation}\label{eq2.2}\psi_{(2^m)}=p_{2m, T}
=\sum_{\stackrel{j=1} {j \text{ odd}}}^{m+1} {{m+1}\choose {j}} h_2^{m+1-j} e_2^{j-1}=h_2 V_{2m-2}+e_2^m \ \mathrm{Odd}(m+1),\end{equation}
where the notation $\mathrm{Odd}(n)$ is used to signify 1 if $n$ is odd, and 0 otherwise, and we have set
\begin{equation}\label{eq2.3}V_{2m-2}=\sum_{\stackrel{j=1} {j \text{ odd}}}^{m} {{m+1}\choose {j}} h_2^{m-j} e_2^{j-1}.\end{equation}

From eqns.~\eqref{eq2.2} and~\eqref{eq2.3}, we need to establish the Schur positivity of 
\begin{equation}\label{eq2.4} h_2 (V_{2m-2}-p_1^{2m-2})+e_2^m \ \mathrm{Odd}(m+1).\end{equation}
In fact when $m$ is odd, we will show that $V_{2m-2}-p_1^{2m-2}$ itself is Schur-positive, whereas when $m$ is even, we will need to multiply this by  $h_2$ and examine the entire expression~\eqref{eq2.4} in order to obtain Schur positivity.

Since $p_1^2=h_2+e_2,$  we can write
\begin{equation}\label{eq2.5}p_1^{2m-2}=\sum_{t=0}^{m-1} {m-1\choose t} e_2^t h_2^{m-t-1}.\end{equation}
Also \begin{align*}
V_{2m-2}&=\sum_{\stackrel{j=1} {j \text{ odd}}}^{m} ({{m}\choose {j}}+ {m\choose j-1}) h_2^{m-j} e_2^{j-1}\\
&=\sum_{\stackrel{t=0} {t \text{ even}}}^{m-1} ({{m}\choose {t+1}}+ {m\choose t}) h_2^{m-t-1} e_2^{t}, \quad (\text{setting }t=j-1)\\
&=\sum_{\stackrel{t=0} {t \text{ even}}}^{m-2} ({{m}\choose {t+1}}+ {m\choose t}) h_2^{m-t-1} e_2^{t} + (m+1)e_2^{m-1} \ \mathrm{Odd}(m).\end{align*}
Combining  this with eqn.~\eqref{eq2.5}, we obtain 
\begin{equation*}V_{2m-2}-p_1^{2m-2}=
\sum_{\stackrel{t=0} {t \text{ EVEN}}}^{m-2} e_2^t h_2^{m-1-t} \left({m-1\choose t+1}
+{m-1\choose t } + {m-1\choose t-1}\right)\end{equation*}
\begin{equation}\label{eq2.6} + (m+1)e_2^{m-1} \ \mathrm{Odd}(m)
-\sum_{\stackrel{t=0} {t \text{ ODD}}}^{m-1}  e_2^t h_2^{m-1-t}{m-1\choose t},
\end{equation}
where by convention ${m-1\choose t-1}$ is zero if $t<1.$ 
We will split the first sum in eqn.~\eqref{eq2.6} (over even $t$) into three sums as follows:
$$\sum_{\stackrel{t=0} {t \text{ EVEN}}}^{m-2}P_{(t+)}
+\sum_{\stackrel{t=0} {t \text{ EVEN}}}^{m-2} Q_{(t)}
+\sum_{\stackrel{t=2} {t \text{ EVEN}}}^{m-2}R_{(t-)},$$
where, for $0\leq t\leq m-2,$  $$P_{(t+)}
= e_2^t h_2^{m-1-t} {m-1\choose t+1},$$
$$Q_{(t)}= e_2^t h_2^{m-1-t} {m-1\choose t},$$ 
and for $2\leq t\leq m-2,$
$$R_{(t-)}= e_2^t h_2^{m-1-t}  {m-1\choose t-1}.$$
Next consider the negated terms in eqn.~\eqref{eq2.6}.  For these we write, for odd $k,$ 
$1\leq k\leq m-1,$  $$N_k=e_2^k h_2^{m-1-k}{m-1\choose k}.$$ 

Our goal is  to absorb every negated term $N_k$ into a Schur-positive term.  We now describe a judicious  grouping which will allow us to accomplish this.

Collect the terms in $V_{2m-2}-p_1^{2m-2}$ as follows:
$$(P_{(0+)}-N_1+R_{(2-)})+ (P_{(2+)}-N_3+R_{(4-)})+\ldots +(P_{(t+)}-N_{t+1}+R_{(t+2)-})+\ldots,$$
where $t$ is even.

For each odd $k=2j-1, 1\leq k\leq m-3,$ we have
$$P_{((2j-2)+)}-N_{2j-1}+R_{((2j)-)}=P_{((k-1)+)}-N_k+R_{((k+1)-)}$$
\begin{align}\label{eq2.7}
&={m-1\choose k} e_2^{k-1} h_2^{m-k-2}(h_2^2-e_2h_2+e_2^2), \quad 1\leq k\leq m-3.\end{align}
But this is Schur-positive by Lemma~\ref{lem2.8}. 
%$h_2^2-e_2h_2+e_2^2=s_{(4)}+s_{(1^4)}+2s_{(2,2)}.$ 
This absorbs the negative terms $N_k$ for $k$ odd, $k\leq m-3,$ into Schur-positive expressions.  Looking at eqn.~\eqref{eq2.6}, there is only one more negated term to investigate.

Suppose  $m$ is {\bf odd}, so that $k=m-2$ is odd, and thus $N_{m-2}$ is the last negated summand in~eqn.\eqref{eq2.6}, the only one not taken care of in eqn.~\eqref{eq2.7}.
Group the terms as before, and noting that  
$P_{((m-3)+)}$ was NOT used in the groupings of eqn.~\eqref{eq2.7}, we have 
\begin{align*}&P_{((m-3)+))}- N_{m-2} +e_2^{m-1} (m+1)\\
&={m-1\choose m-2} e_2^{m-3}h_2^2-{m-1\choose m-2} e_2^{m-2} h_2+ (m+1) e_2^{m-1}\\
 &=(m-1)e_2^{m-3}  (h_2^2-e_2 h_2 +e_2^2) +2 e_2^{m-1}
\end{align*}
and this is again Schur-positive by Lemma~\ref{lem2.8}.  (Admittedly this is something of a miracle.)

Next suppose  $m$ is {\bf even}, so that $k=m-1$ is odd; then  $N_{m-1}$ is the last negated summand in eqn.~\eqref{eq2.6}, and the only one not absorbed in eqn.~\eqref{eq2.7}.
Now we have (since $P_{((m-2)+)}$ was not used in any of the  groupings in eqn.~\eqref{eq2.7}):
$$P_{((m-2)+)}-N_{m-1}={m-1\choose m-1} e_2^{m-2} h_2 -{m-1\choose m-1} e_2^{m-1} =e_2^{m-2}(h_2-e_2).$$
While this is not itself Schur-positive, from eqn.~\eqref{eq2.4} we can multiply by $h_2,$ and then we have (again somewhat fortuitously), since $m+1$ is odd,
 $$ h_2(P_{((m-2)+}-N_{m-1})+e_2^m\ \mathrm{Odd}(m+1)
=e_2^{m-2}(h_2^2-h_2 e_2+e_2^2),$$ 
and this is Schur-positive as before.   

By eqn.~\eqref{eq2.4}, this completes the argument, in which Lemma~\ref{lem2.8} clearly played a crucial role.  \end{proof}

Next we examine the chain from the top down.  Here the arguments are more direct, but also computationally technical. Our strategy for establishing Schur positivity will be to show that  
in the Schur expansion of the sum $\sum_{\nu\vdash n:\mu<\nu\leq (n)} p_\nu,$ i.e. starting from the partition $(n)$ and moving down the chain, the positive multiplicities of the irreducibles never exceed those in $\psi_n.$ For the cases we consider, we are able to show that, with the exception of the trivial representation, which clearly occurs as many times as the number of partitions in the interval $[\mu, (n)],$ this sum has multiplicities at most 4.  This allows us to apply the lower bounds for the multiplicities in $\psi_n$ developed earlier in Lemma~\ref{lem2.6}.
The Schur function expansion of the product $p_np_m$ figures prominently in this analysis.  

In the proofs that follow,  for simplicity and clarity we write $\lambda$ for the Schur function $s_\lambda$ indexed by $\lambda.$  The context should make clear when $(n-2,2)$  indicates the Schur function $s_{(n-2,2)}$ rather than the partition itself.  Also for any statement $S,$  $\delta_S$ denotes the value 1 if and only if $S$ is true,  and is zero otherwise.

Recall \cite{M, St4EC2} that $\omega$ is the involution defined on the ring of symmetric functions  which sends $h_n$ to $e_n.$ Thus $\omega$ takes the Schur function indexed by a partition $\lambda$ to the Schur function indexed by the conjugate partition $\lambda^t.$ (In the character table of $S_n$, it corresponds to tensoring with the sign representation.)  

%%%%%%%%%%%%%%%%%%%%%%%%
%More general lemma on p_n p_m

\begin{lemma}\label{lem2.12} Let 
 $n\geq m$ and $4\geq m\geq 1;$ if $n=m$ assume $m\neq 2.$ Then the Schur function expansion of the product $p_n p_m$ has only the coefficients $0, \pm 1.$ More precisely,  one has the following expansion into distinct irreducibles:
 \begin{align*}&D_{n,m}\\
 &+[\alpha(n,m)+\delta_{m\geq 2}\beta(n,m)+\delta_{m=3}\gamma(n,3)+\delta_{m=4}\gamma(n,4)]\\
 &+ (-1)^{n+m}\omega\left( [\alpha(n,m)+\delta_{m\geq 2}\beta(n,m) 
 +\delta_{m=3}  \gamma(n,3)
 +\delta_{m=4}\gamma(n,4)]\right); \end{align*}
 where 
 \begin{align*}& D_{n,m}=\sum_{s=0}^{m-1} \sum_{t=0}^{n-m-2}(-1)^{t+1} (n-t-s-1, m-s+1, 2^s, 1^t),\\
 &\alpha(n,m)=\sum_{r=0}^{m-1} (-1)^r (n-r+m, 1^r) ,  \quad
 \beta(n,m)=\sum_{s=0}^{\bf m-2}  (-1)^s (n,m-s,1^s),\\
 &\gamma(n,3)=(n-1,2^2), \text{ and }
 \gamma(n,4)=(n-1,3,2)-(n-1,2,2,1)+(n-2,2^3).
 \end{align*}
 The number of irreducibles appearing in the expansion is 
 $$m(n-m+1)+2(\delta_{m\geq 2}(m-1)+\delta_{m=3}+3\delta_{m=4}).$$
 \begin{comment}
 %rewritten on 2018 April 7
\begin{equation*}\delta_{m=3}  \left((n-1,2,2) +(-1)^{n-1} (3,3,1^{n-3})\right)\end{equation*}
\begin{equation*} +\delta_{m=4}\left((n-1,3,2)-(n-2,3,2,1)+(n-2,2^3)\right)
\end{equation*}
\begin{equation*}+\delta_{m=4}\left((-1)^n [(3^2,2, 1^{n-4}) - (4,3,1^{n-4}) +(4^2, 1^{n-4})] \right)        \end{equation*}
\begin{equation*}+ \delta_{m\geq 2}\left(\sum_{s=0}^{\bf m-2}  (-1)^s (n,m-s,1^s)
+\sum_{s=\bf 1}^{m-1} (-1)^{n-s-1}(m-s+1, 2^s, 1^{n-s-1})\right)
\end{equation*}
\begin{equation*}+\sum_{r=0}^{m-1} (-1)^r (n-r+m, 1^r) +\sum_{t=n}^{n+m-1} (-1)^{t-1} (n-t+m, 1^t)  + D_{n,m},\end{equation*}
%
where $D_{n,m}$ is the following signed sum of double hooks:
%
$$D_{n,m}=\sum_{s=0}^{m-1} \sum_{t=0}^{n-m-2}(-1)^{t+1} (n-t-s-1, m-s+1, 2^s, 1^t).$$
\end{comment}
%
 \end{lemma}
 
 \begin{proof}  Note that these definitions imply that 
 $$(-1)^{n+m}\omega(\alpha(n,m))=\sum_{t=n}^{n+m-1} (-1)^{t-1} (n-t+m, 1^t)$$
 and 
 $$(-1)^{n+m}\omega(\beta(n,m))=\sum_{s=\bf 1}^{m-1} (-1)^{n-s-1}(m-s+1, 2^s, 1^{n-s-1}).$$
   We begin with the well-known expansion of $p_n$ into Schur functions of hook shape (a special case of the Murnaghan-Nakayama rule):
 $$p_n=\sum_{r=0}^{n-1} (-1)^r (n-r, 1^{r}), n\geq 2.$$
 
The Murnaghan-Nakayama rule says the Schur functions  in the product $p_n p_m$ are indexed by partitions obtained by attaching border strips (or rim hooks) of size $m$ to each of the above hooks, with sign $(-1)^s $ where $s$ is one less than the number of rows occupied by the border strip.  (See \cite{M} or \cite{St4EC2}.)    We enumerate the disjoint possibilities in the figures below.   Note that we have excluded the  partition $(1^m)$ (respectively, $(m)$) from  Figure 1a because it is counted in Figure 2b (respectively, Figure 2a).  In particular  Figures 1a-1b need to be considered separately only if $m\geq 2.$  (When $m=1,$ Figure 1a is included as the special case of Figure 2b  for $r=0,$ and similarly Figure 1b is the $r=n-1$ case of Figure 2a.) Figures 2-3 are possible configurations for all $m\geq 1,$  and Figures 4a-4b, 5a-5d can occur only if $m=3,4$ respectively.

\vskip .1in

\ytableausetup{centertableaux}
\begin{center} \begin{tiny}
\begin{ytableau}
  \phantom{X}&  & &  &  & \\
X  &X  &\ldots & X \\
X\\
\ldots \\
X
\end{ytableau}\end{tiny} 
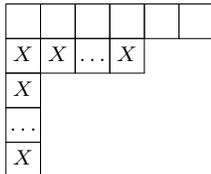
\begin{small} \begin{tiny}{\bf Figure 1a:} $\mu=(n),$ with hook $(m-s, 1^s), 0\leq s\leq {\bf m-2},$   in row 2. \end{tiny} \\The contribution to $p_np_m$ here is $(-1)^s (n,m-s,1^s),$ if $m\geq 2.$\end{small}
\end{center}
%The contribution to $p_np_m$ here is $(-1)^0 (n,m),$ if $m\geq 2.$
\vskip .1in
\ytableausetup{centertableaux}
\begin{center} \begin{tiny}
\begin{ytableau}
  \phantom{X} &X &\ldots &X\\
 &X \\
  &\ldots \\
 &X\\
 \\
 \\
\end{ytableau}\end{tiny} 
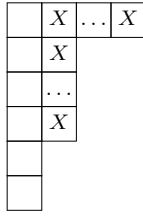
\begin{small} \begin{tiny}{\bf Figure 1b:} $\mu=(1^n),$ with hook $(m-s, 1^s), {\bf 1}\leq s\leq m-1, $   in column 2. \end{tiny} \\ The contribution to $p_np_m$  is $(-1)^{n-1+s} (m-s+1,2^s, 1^{n-s-1}),$  if $m\geq 2.$\end{small}
\end{center}
%The contribution to $p_np_m$  is $(-1)^{n-1+m-1} (2^m, 1^{n-m}),$  if $m\geq 2.$

\vskip .1in
\ytableausetup{centertableaux}
\begin{center} \begin{tiny}
\begin{ytableau}
  \phantom{X}&  & &  &  & &X &X &\ldots & X \\
  \\
  \\
 \\
\end{ytableau}\end{tiny} 
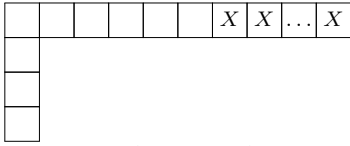
\begin{small} \begin{tiny}{\bf Figure 2a:} $\mu=(n-r, 1^r), 0\leq r\leq n-1,$ with horizontal strip of size $m.$\end{tiny} \\The contribution to $p_np_m$ here is $(-1)^{r} (n-r+m, 1^r), 0\leq r\leq n-1.$
\end{small}
\end{center}
%The contribution to $p_np_m$ here is $(-1)^{r} (n-r+m, 1^r), 0\leq r\leq n-1.$

\vskip .1in
\ytableausetup{centertableaux}
\begin{center} \begin{tiny}
\begin{ytableau}
  \phantom{X}&  & &  &  &  \\
  \\
  \\
 X\\
 X\\
 \ldots\\
 X\\
\end{ytableau}\end{tiny} 
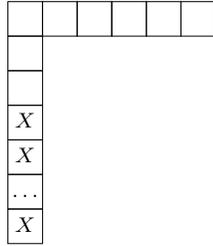
\begin{small} \begin{tiny}{\bf Figure 2b:} $\mu=(n-r, 1^r), 0\leq r\leq n-1,$ with vertical strip of size $m.$\end{tiny} \\ The contribution to $p_np_m$ here is $(-1)^{r+m-1} (n-r, 1^{r+m}), 
0\leq r \leq n-1.$ \end{small}
\end{center}
\vskip .1in
%The contribution to $p_np_m$ here is $(-1)^{r+m-1} (n-r, 1^{r+m}), 
%0\leq r \leq n-1.$
%\vskip .1in

\ytableausetup{centertableaux}
\begin{center} \begin{tiny}
\begin{ytableau}
  \phantom{X}&  & &  &  &  &\\
  &X &X &\ldots & X\\
  &X \\
 &\ldots\\
 &X\\
  \\
  \\
\end{ytableau}\end{tiny} 
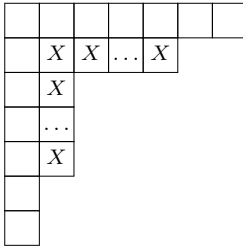
\begin{small} \begin{tiny}{\bf Figure 3:} $\mu=(n-r, 1^r), r\geq 1, n-r-1\geq 1$ with hook $(m-s, 1^s), 0\leq s\leq m-1,$ attached. \end{tiny}\\ The contribution to $p_np_m$  is  $(-1)^{r+s} (n-r, m-s+1, 2^s,1^{r-1-s}), 0\leq s\leq m-1, 1\leq r\leq n-2.$\end{small}
\end{center}
%The contribution to $p_np_m$  is  $(-1)^{r+s} (n-r, m-s+1, 2^s,1^{r-1-s}).$
% in Frobenius notation: $(-1)^{r+s} (n-r-1>m-s-1|r>s).$

If $m=3$ we have a conjugate pair of additional configurations: 
\vskip .1in

\ytableausetup{centertableaux}
\begin{center} \begin{tiny}
\begin{ytableau}
  \phantom{X}&  & &  &  & \\
  &X\\
    X & X \\
\end{ytableau}\end{tiny} 
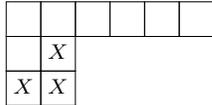
\begin{small} \begin{tiny}{\bf Figure 4a:} $\mu=(n-1,1),$ with rim hook of 
size $m=3. $\end{tiny}\\  The contribution to $p_np_3$ is $(-1)^2 (n-1,2,2).$ \end{small}
\end{center}
%The contribution to $p_np_3$ here is $(-1)^0 (n-1,2,2).$ 
\vskip .1in
\ytableausetup{centertableaux}
\begin{center} \begin{tiny}
\begin{ytableau}
  \phantom{X} & &X\\
 &X &X\\
  \\
 \\
 \\
 \\
\end{ytableau}\end{tiny} 
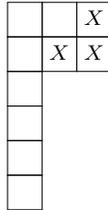
 \begin{small}\begin{tiny} {\bf Figure 4b:} $\mu=(2,1^{n-2}),$ with rim hook of size $m=3.$\end{tiny} \\ The contribution to $p_np_3$  is $(-1)^{n-2}\cdot (-1) (3,3, 1^{n-3}).$  
 \end{small}
\end{center}
%The contribution to $p_np_3$  is $(-1) (3,3, 1^{n-6}).$  
%
If $m=4$ we have six additional  configurations, which we list in successive conjugate pairs: 
%\vskip .1in
%
\ytableausetup{centertableaux}
\begin{center} \begin{tiny}
\begin{ytableau}
  \phantom{X}&  & &X   \\
                   &X &X &X\\
     \\
     \\
     \\
\end{ytableau}\end{tiny} 
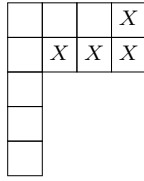
\begin{small} \begin{tiny}Figure 5a: $\mu=(3,1^{n-3}),$ with rim hook of 
size $m=4. $\end{tiny}\\  The contribution to $p_np_4$ is $(-1)^{n-3}
\cdot (-1) (4,4, 1^{n-4}).$ \end{small}
\end{center}
\vskip .1in
\ytableausetup{centertableaux}
\begin{center} \begin{tiny}
\begin{ytableau}
  \phantom{X} & & \ldots &\\
 &X \\
 &X  \\
X &X 
\end{ytableau}\end{tiny} 
 \begin{small}\begin{tiny}{\bf Figure 5b:} $\mu=(n-2,1^{2}),$ with rim hook of size $m=4.$\end{tiny} \\ The contribution to $p_np_4$  is $(-1)^{2}\cdot (-1)^2 (n-2,2,2,2).$  
 \end{small}
 \end{center}
\vskip .1in
\ytableausetup{centertableaux}
\begin{center} \begin{tiny}
\begin{ytableau}
  \phantom{X}&  & &  &  & \\
  &X &X\\
    X & X \\
\end{ytableau}\end{tiny} 
\begin{small}\begin{tiny}{\bf  Figure 5c:} $\mu=(n-1,1),$ with rim hook of 
size $m=4. $\end{tiny}\\  The contribution to $p_np_4$ is $(-1)^2 (n-1,3,2).$ \end{small}
\end{center}
\vskip.1in
\ytableausetup{centertableaux}
\begin{center} \begin{tiny}
\begin{ytableau}
  \phantom{X} & &X \\
  &X &X\\
 &X  \\
 \\
 \\
 \\
\end{ytableau}\end{tiny} 
 \begin{small}\begin{tiny}{\bf Figure 5d:} $\mu=(2,1^{n-2}),$ with rim hook of size $m=4.$\end{tiny} \\ The contribution to $p_np_4$  is $(-1)^{n-2}\cdot (-1)^2 (3,3,2, 1^{n-4}).$  
 \end{small}
\end{center}

% two more configurations 
\vskip .1in
\ytableausetup{centertableaux}
\begin{center} \begin{tiny}
\begin{ytableau}
  \phantom{X}&  & &  &  & \\
   &X\\
    X & X \\
    X
\end{ytableau}\end{tiny} 
\begin{small}\begin{tiny}{\bf  Figure 5e:} $\mu=(n-1,1),$ with rim hook of 
size $m=4. $\end{tiny}\\  The contribution to $p_np_4$ is $(-1)\cdot (-1)^2 (n-1,2,2,1).$ \end{small}
\end{center}
%The contribution to $p_np_3$ here is $(-1)^0 (n-1,2,2).$ 
\vskip .1in
\ytableausetup{centertableaux}
\begin{center} \begin{tiny}
\begin{ytableau}
  \phantom{X} & &X &X\\
 &X &X\\
  \\
 \\
 \\
 \\
\end{ytableau}\end{tiny} 
 \begin{small}\begin{tiny}{\bf Figure 5f:} $\mu=(2,1^{n-2}),$ with rim hook of size $m=4.$\end{tiny} \\ The contribution to $p_np_4$  is $(-1)^{n-2}\cdot (-1) (4,3, 1^{n-4}).$  
 \end{small}
\end{center}

The sum of hooks in Figures 2a-2b collapses as follows:
\begin{align*} &\sum_{r=0}^{n-1} (-1)^r (n-r+m, 1^r) +
\sum_{r=0}^{n-1} (-1)^{r +m-1}(n-r, 1^{r+m});  \quad (*)\\
&\text{if }n-1\geq m, \text{ we can split the first sum:}\\
&=\sum_{r=0}^{m-1} (-1)^r (n-r+m, 1^r) +\sum_{r=m}^{n-1} (-1)^r (n-r+m, 1^r) 
+\sum_{t=m}^{n+m-1} (-1)^{t-1}(n+m-t, 1^{t})\\
&=\sum_{r=0}^{m-1} (-1)^r (n-r+m, 1^r) +(\sum_{r=m}^{n-1} (-1)^r (n-r+m, 1^r) +\sum_{t=m}^{n+m-1} (-1)^{t-1}(n+m-t, 1^{t}))\\
&=\sum_{r=0}^{m-1} (-1)^r (n-r+m, 1^r) +\sum_{t=n}^{n+m-1} (-1)^{t-1} (n-t+m, 1^t).
\end{align*}
If $n=m,$ this latter expression coincides with $(*)$ above, and is thus valid for all $n\geq m.$ 
% Removed 2018-12-16:The first two lines in the statement of the lemma come from Figures 4a-4b and 5a-5f respectively.  The first two summations come from Figures 1a-1b, and the third and fourth summations are the result of the collapse  between Figures 2a and 2b. 

%
Note that in the double hook %$(n-r-1, m-s-1|r,s)$ 
of Figure 3, we must have $n-r\geq m-s+1\geq 1$ and similarly $r>s\geq 0.$ This gives $0\leq s\leq m-1$ and $s<r<n-m+s$.
Hence Figure 3 contributes the sum of double hooks
\begin{align*}
D_{n,m}&= \sum_{s=0}^{m-1} \sum_{\stackrel{r}{s<r<n-m+s}}
(-1)^{r+s} (n-r, m-s+1, 2^s,1^{r-1-s})\\
&=\sum_{s=0}^{m-1} \sum_{t=0}^{n-m-2}(-1)^{t+1} (n-t-s-1, m-s+1, 2^s, 1^t),  \\
&\text{where we have put $t=r-1-s$.}
\end{align*}
%
%Finally note that the partition $(n-r-1, m-s-1|r,s)$ is Frobenius notation corresponds to the partition of $n+m$ whose parts are  $(n-r, m-s+1, 2^s,1^{r-1-s}).$
%
Figures 1a-1b contribute the sum
\begin{equation*} \delta_{m\geq 2}\left(\sum_{s=0}^{\bf m-2}  (-1)^s (n,m-s,1^s)
+\sum_{s=\bf 1}^{m-1} (-1)^{n-s-1}(m-s+1, 2^s, 1^{n-s-1})\right);
\end{equation*}
putting $\beta(n,m)$ for the first sum above, we see that this can by rewritten as 
 $$\delta_{m\geq 2}\left(\beta(n,m) +(-1)^{n-m}\omega(\beta(n,m)\right).$$
 (Check that $\omega(n, m-s, 1^s)=(s+2, 2^{m-s-1}, 1^{n-m+s}),$ and make the substitution $t=m-s-1$ to get the second sum multiplied by $(-1)^{n-m}$.)
 
 Figures 2a-2b contribute
\begin{equation*}+\sum_{r=0}^{m-1} (-1)^r (n-r+m, 1^r) +\sum_{t=n}^{n+m-1} (-1)^{t-1} (n-t+m, 1^t);\end{equation*} 
putting $\alpha(n,m)$ for the first sum of $m$ hooks above, we that this can be rewritten as 
$$\alpha(n,m)+ (-1)^{n+m} \omega(\alpha(n,m)).$$

Putting $\gamma(n,3)=(n-1,2^2),$ the contribution of Figures 4a-4b is seen to be 
$$\delta_{m=3}  \left(\gamma(n,3) +(-1)^{n+m} \omega(\gamma(n,3))\right),  \text{ since, when }m=3,  (-1)^{n-1}=(-1)^{n+m}$$
and similarly setting $\gamma(n,4)=(n-1,3,2)-(n-2,3,2,1)+(n-2,2^3),$ the contribution   of Figures 5a-5f is
$$\delta_{m=4}\left(\gamma(n,4)+(-1)^{n+m} \omega(\gamma(n,4))\right) 
\text{ again since, when }m=4,  (-1)^{n}=(-1)^{n+m}.$$
This completes the proof.
 \end{proof}

We now specialise this lemma to the values $m\leq 4.$  In what follows it will be convenient to write partitions of $n$ as $(*, \mu),$ where $*$ will indicate a single part equal to $n-|\mu|,$ and $\mu$ is a partition whose largest part does not exceed the part indicated by $*.$ For example, $(*, 2,2,1^t)$ means the partition $(n-4-t, 2,2,1^t).$
\begin{lemma}\label{lem2.13} One has the following Schur function expansions: 
\begin{enumerate}[label=(\arabic*)] \item 
$p_n=(n)+(-1)^{n-1}(1^n)+\sum_{r=1}^{n-2} (-1)^r (*, 1^{r})$ for $n\geq 2$
(the summation is nonzero if and only if $n\geq 3$);
\item $p_1p_{n-1} $ 
$=(n)+(-1)^n (1^n) +
\sum_{r=0}^{n-4} (-1)^{r-1} (*, 2,1^r)$ for $n\geq 3$
( note the summation is nonzero if and only if $n\geq 4$);
\item \noindent   $  p_{n-2} p_2$  (for $n\geq 5$)
\begin{align*}&=(n)-(n-1,1)+(n-2,2) +(-1)^n[(1^n)-(2,1^{n-2})+(2^2,1^{n-4})]\\
& +\sum_{t=0}^{n-6} (-1)^{t-1}(*, 3, 1^t)
+\sum_{t=0}^{n-6} (-1)^{t-1}(*, 2,2,1^t).
\end{align*}
(Note that the summations are nonzero if and only if $n\geq 6$).
\item $p_{n-2} h_2  $  (for $n\geq 5$) 
\vskip.07in
$=(n) +(-1)^n (2^2, 1^{n-4}) 
+ (-1)^{n-1} (2, 1^{n-2})
+\sum_{r=0}^{n-6}(-1)^{r+1} (*,3,1^{r}).$
\vskip.07in
%%%% p_{n-2} p_1^2 moved back here from below 2018-6-18
 \item $p_{n-2} p_1^2  $  (for $n\geq 5$)
\begin{align*}&=(n) +(n-1,1)-(n-2,2)
 -(-1)^n (1^n) +(-1)^n (2^2, 1^{n-4}) +(-1)^{n-1}(2,1^{n-2})\\
&+\sum_{r=0}^{n-6}(-1)^{r+1} (*,3,1^{r})
 +\sum_{t=0}^{n-6} (-1)^{t}(*, 2,2,1^t)\end{align*}
\item $p_{n-3} p_3$  (for $n\geq 6$) 
\begin{align*}
&=[(n)-(n-1,1)+(n-2,1^2)+(n-3,3)-(n-3,2,1)]+(n-4,2^2)\\
&+ (-1)^n[(1^n)-(2,1^{n-2})+(3, 1^{n-3})+(2^3,1^{n-6})-(3,2,1^{n-5})]+(-1)^{n-4}(3^2,1^{n-6})\\
& +\sum_{t=0}^{n-8} (-1)^{t-1}(*, 4, 1^t) 
+\sum_{t=0}^{n-8} (-1)^{t-1} (*, 3,2,1^t)+\sum_{t=0}^{n-8} (-1)^{t-1} (*, 2,2,2,1^t) 
\end{align*}
\item $p_{n-3} p_2 p_1$
\begin{align*}
&=(n)-(n-2,1^2)+(n-3,2,1)-(n-4,2^2)\\
&+(-1)^{n-1}[(1^n)-(3,1^{n-3})+(3,2,1^{n-5})-(3^2, 1^{n-6})] \\
&+\sum_{t=0}^{n-8} (-1)^{t-1} (*, 4, 1^t)
+\sum_{r=0}^{n-8} (-1)^{r} (*, 2,2,2,1^{r}) 
%&\sum_{t=1}^{n-7} (-1)^{t-1} (n-5-t, 2,2,2,1^{t-1}) \quad (B4)
\end{align*}
%%%%%%%%%%%(p_{n-3} h_2 p_1 corrected 8 April 2018)
\item $p_{n-3} h_2 p_1$ 
\begin{align*}&= (n)+(n-1,1)-(n-3,3) +(-1)^{n-2}(2,1^{n-2})+(-1)^{n-1}(2^3,1^{n-6})\\
& +(-1)^{n-1} (3,2,1^{n-5}) +(-1)^n (3,1^{n-3}) +(-1)^{n-6}(3^2,1^{n-6})\\
&+\sum_{r=0}^{n-8}(-1)^{r+1} (*, 4, 1^r)+\sum_{t=0}^{n-8} (-1)^t (*,3,2,1^t)
\end{align*}
%%%%(corrected 8 April 2018)
\begin{comment}
\item $p_{n-3}\psi_3=p_{n-3}(p_3+2h_2p_1)$  NEED TO CORRECT THIS
\begin{align*} &=3(n)+(n-1,1)+(n-2,2)-(n-3,3)+(n-4,2^2)+(-1)^n(1^n)\\
&+ (-1)^{n-1}(2,1^{n-2})  +3(-1)^n (3,1^{n-3})+(-1)^n (3^2, 1^{n-6})\\
%+(-1)^n(2^3, 1^{n-6}) cancels t=0 term in last sum below
&+3 \sum_{t=0}^{n-8} (-1)^{t+1} (*, 4, 1^t)+\sum_{t=0}^{n-8} (-1)^{t-1} (*, 3,2, 1^t)
+\sum_{t=1}^{n-8} (-1)^{t-1} (*, 2^3, 1^t)
%t=0 term cancels (-1)^n(2^3, 1^{n-6}) above
\end{align*}
\end{comment}
\item $p_{n-4}p_4$
\begin{align*}&=(n)-(n-1,1)+(n-2,1^2)-(n-3,1^3)\\
&+(*,4)-(*,3,1)+(*,2,1^2)+(*,3,2)-(*,2^2,1)+(*,2^3)\\
&+(-1)^n[(1^n)-(2,1^{n-2})+(3,1^{n-3})-(4,1^{n-4})+(2^4,1^{n-8})\\
&\phantom{+(-1)^n[}-(3,2^2,1^{n-7})+(4,2,1^{n-6})+(3^2,2,1^{n-8})-(4,3,1^{n-7})+(4^2,1^{n-8})]\\
&+\sum_{t=0}^{n-10}(-1)^{t+1}\left[(*,5,1^t)+(*,4,2,1^t)+(*,3,2^2,1^t)+(*,2,2^3,1^t)\right]
\end{align*}
 
\item $p_{n-4} p_3 p_1$ %checked for n=8 to 12
\begin{align*}&=(n)-(n-2,2)+(n-3,3)+(n-3,1^3)-(n-4,2,1^2)+(n-5,2^2,1)\\
&+(-1)^{n-1}[(1^n)-(2^2,1^{n-4})+(2^3,1^{n-6})+(4,3,1^{n-7})-(4,2,1^{n-6})+(4,1^{n-4})]\\
&+(-1)^n(4^2,1^{n-8})
-(n-6,2^3)\\
&+\sum_{t=0}^{n-10} (-1)^{t-1}(*,5,1^t)+
\sum_{t=0}^{n-9} (-1)^{t-1}(*,3^2,1^t)+\sum_{t=1}^{n-9} (-1)^{t-1}(*,2^4,1^{t-1}) 
\end{align*}

\begin{comment} %collapsed to above, checks for n>=8
\begin{align*}&=(n)-(n-2,2)+(n-3,3)+(n-3,1^3)+(n-4,4)-(n-4,2,1^2)\\
&+(n-5,3,2)+(n-5,2^2,1)\\
&+(-1)^{n-1}[(1^n)-(2^2,1^{n-4})+(2^3,1^{n-6})+(2^4,1^{n-8})\\
&\phantom{+(-1)^{n-1}[}+(3^2,2,1^{n-8})+(4,3,1^{n-7})-(4,2,1^{n-6})+(4,1^{n-4})]\\
&+\delta_{n\geq 9}[ -(n-4,4)+(-1)^n(4^2,1^{n-8})-(n-5,3,2)\\
&\phantom{\delta_{n\geq 9}}
+(-1)^n (3^2,2,1^{n-8})-(n-6,2^3)+(-1)^n(2^4,1^{n-8}) ]\\
&+\sum_{t=0}^{n-10} (-1)^{t-1}(*,5,1^t)+
\sum_{t=0}^{n-9} (-1)^{t-1}(*,3^2,1^t)+\sum_{t=1}^{n-9} (-1)^{t-1}(*,2^4,1^{t-1}) 
\end{align*}
\end{comment}

\item $p_{n-4} p_2 p_1^2$ %checked for n=8 to 11
\begin{align*}&=(n)+(n-1,1)-(n-2,1^2)-(n-3,1^3) -(n-4,4)\delta_{n\ge 9} +(n-4,3,1)\\ 
& +(n-4,2,1^2)-(n-5,3,2) -(n-5,2^2,1)+(n-6,2^3) \\
& +(-1)^n[(1^n)+(2,1^{n-2})-(2^4,1^{n-8}) -(3,1^{n-3}) +(3,2^2,1^{n-7})-(3^2,2,1^{n-8}) ]\\
& +(-1)^n [(-(4,1^{n-4})+(4,2,1^{n-6})-(4,3,1^{n-7}) +(4^2,1^{n-8}) ]\\
& +\delta_{n\ge 10}\sum_{t=0}^{n-10} (-1)^{t-1} \{(*,5,1^t)-(*,3,2^2,1^t)\} 
 +\delta_{n\ge 10}\sum_{r=0}^{n-10} (-1)^{r}\{(*4,2,1^{r}) -(*,2^4,1^{r})\}
%& +\delta_{n\ge 9}\sum_{t=0}^{n-9} (-1)^{t-1} \{(*,5,1^t)-(*,3,2^2,1^t)\}  
% +\delta_{n\ge 9}\sum_{t=1}^{n-9} (-1)^{t-1}\{(*4,2,1^{t-1}) -(*,2^4,1^{t-1})\}  --reindexed; note n-5-t\ge 5 implies t\le n-10 etc.
\end{align*}

\end{enumerate}
%%%%%%%%%%%%%%%%%%%%%%%%%%%

%%%%%%%%%%%%%%%%%%%%%%%%%%%%
\end{lemma}
\begin{proof} {\bf Part (1)} is the well-known expansion of the power sum $p_n$ into Schur functions indexed by hooks  (see, e.g. \cite{M}). 

{\bf Parts (2)-(3)} and {\bf  Part (6)}, {\bf  Part (9)} follow by putting $m=1,2,3,4$  in Lemma~\ref{lem2.12}, and replacing 
$n$ with $n-1, n-2, n-3, n-4$ respectively.  In {\bf Part (6)}, note that $\gamma(n-3,3) =(n-4,2,2)$ and hence the corresponding term is 
$(n-4,2,2)+(-1)^{n-4} (3,3,1^{n-6}).$ 

The expansion of $p_{n-2}p_1^2$ in {\bf Part (5)} follows by subtracting twice the equation in  {\bf Part (3)} from {\bf Part (4)}, by virtue of the identity $p_1^2=2h_2-p_2.$

For {\bf Part (4)}:

Write $p_{n-2}=(n-2) + \sum_{r=1}^{n-5} (-1)^r (n-2-r, 1^r) + (-1)^{n-4} (2, 1^{n-4}) +(-1)^{n-3} (1^{n-2}).$

Using the Pieri rule, we have
\begin{align*}&p_{n-2}h_2=(n)+(n-1,1) +(n-2,2) \\
&+(-1)^{n-4}\left( (4, 1^{n-4}) +( 3,2,1^{n-5}) +(3, 1^{n-3}) +(2^2, 1^{n-4}) \right) \\
&+(-1)^{n-3}\left((3, 1^{n-3})+(2, 1^{n-2})  \right)
+\sum_{r=1}^{n-5} (-1)^r (n-2-r, 3, 1^{r-1})\\
&+ \sum_{r=1}^{n-5} (-1)^r (n-1-r, 2, 1^{r-1}) +\sum_{r=1}^{n-5} (-1)^r (n-2-r, 2, 1^r)\quad (F)\\ 
&+\sum_{r=1}^{n-5} (-1)^r (n-r, 1^r) +\sum_{r=1}^{n-5} (-1)^r (n-1-r, 1^{r+1}). \quad (G)
\end{align*}
Line $(F)$ is a telescoping sum which collapses into
$$\sum_{t=0}^{n-6} (-1)^{t+1} (n-2-t, 2, 1^{t}) +\sum_{r=1}^{n-5} (-1)^r (n-2-r, 2, 1^r) =-(n-2,2)+(-1)^{n-5} (3,2,1^{n-5}).$$
Similarly line $(G)$ collapses into
$$\sum_{r=1}^{n-5} (-1)^r (n-r, 1^r) +\sum_{t=0}^{n-4} (-1)^{t-1} (n-t, 1^{t})=-(n-1,1)+(-1)^{n-5} (4,1^{n-4}).$$
Hence $p_{n-2} h_2$ reduces to 
$$(n)+\sum_{r=1}^{n-5} (-1)^r (n-2-r, 3, 1^{r-1}) +(-1)^n (2^2, 1^{n-4}) +(-1)^{n-1}(2, 1^{n-2}).$$

For {\bf Part (7):} We start with the expression for $p_{n-3} p_2$ and multiply by $p_1.$  One checks that this gives 
\begin{align*} &(n)-(n-2,1^2)+(n-3,3)+(n-3,2,1)\\
&+ (-1)^{n-1}(1^n)+(-1)^{n-1}(2^3, 1^{n-6})  +(-1)^{n-2} (3, 1^{n-3}) +(-1)^{n-1} (3,2,1^{n-5})\\
&+\sum_{t=0}^{n-7} (-1)^{t+1} [(*, 3, 1^t)+(*, 4,1^t)+(*, 3,2,1^{t-1}) +(*, 3, 1^{t+1})] \quad (A1)\\
&+\sum_{t=0}^{n-7} (-1)^{t+1}[(*, 2^2, 1^t) + (*, 3,2,1^t)+(*, 2^3, 1^{t-1}) + (*, 2^2, 1^{t+1}) ]\quad(A2)
\end{align*}
Note that the sums in (A1) and (A2) vanish identically unless $n\geq 7.$ The first and last summands in (A1) collapse to $(-(n-3,3)+(-1)^{n-6}(3^2, 1^{n-6}),$ and similarly the first and last summands in (A2) collapse to $-(n-4,2^2)+(-1)^{n-6} (2^3, 1^{n-6}).$ 
Also note that where we have written $1^{t-1}$ for part 1 with multiplicity $t-1,$ there is no contribution unless $t\geq 1.$ 
So the third sum in (A1) and the second sum in (A2) cancel each other. 

Likewise, {\bf Parts (8) and (10)} follow respectively from 
{\bf Parts (4) and (6)}.   Finally {\bf Part (11)} follows from {\bf Part (7)}.
\end{proof}
%%%I don't need this any more?  2018-6-18

%%%   For {\bf Part (7):} We have 
%%%   $\psi_3=p_3+p_2p_1+p_1^3=p_3+p_1(p_2+p_1^2)
%%%  =p_3+2p_1h_2.$  The result follows by combining Part (5) %%%  with twice Part (6).  Note that 
%%% the $t=0$ term in $\sum_{t=0}^{n-8} (-1)^{t-1} (n-6-t, 
%%%   2^3, 1^t)$ cancels the term $(-1)^n(2^3, 1^{n-6})$ in %%%% the expansion of $p_{n-3} p_3.$ 

%%% I don't need this any more?  2018-6-18

\begin{lemma}\label{lem2.14}  The expansion of $p_2^2 p_{n-4}$ is as follows:
%CHECKS with data for n=8,9,10,11,12

\begin{enumerate}[label=(\arabic*)]
\item  The  irreducibles appearing with multiplicity $\pm 2$ are indexed by the following partitions $\lambda:$
\begin{enumerate}
\item $(n-2,2)$ with coefficient 2 and $(n-3,3)$ with coefficient $-2;$
\item $(2^2, 1^{n-4})$ with coefficient $2(-1)^{n-1};$
\item $(2^3, 1^{n-6})$ with coefficient $2(-1)^{n};$
\item $(*,3^2,1^t)$ with coefficient $2(-1)^{t};$
\end{enumerate}
 The remaining irreducibles (with multiplicity) are:

\item $(n), (-1)^{n}(4^2, 1^{n-8}),   
(-1)^{n-1}(4,3,1^{n-7}), (-1)^{n} (4,2,1^{n-6}), (-1)^{n-1} (4,1^{n-4}),$

$ (-1)^{t-1}(*,5,1^t);$

\item $(-1)^{n-1} (1^n), (n-3, 1^3),  (-1)(n-6, 2^3), 
(-1) (n-4,2,1^2),(n-5,2^2,1),(-1)^{t}(*, 2^4, 1^t)  ;$

\item $(-1)^n (3, 1^{n-3}),  (-1) (n-1,1),(-1)^t (3,2^2,1^t),
(-1)^t (3^2,2,1^t), (-1)^{t-1}(*,4,2,1^t), (n-4,4).$

\item $(-1)(n-2,1^2), (-1)^n (2,1^{n-2}),(-1)^{n-1} (2^4,1^{n-8}),   (n-4,3,1), (-1)(n-5, 3, 2),$
$ (-1)^t (*, 3,2^2,1^t).$

%%%%%%%%rewriting above to check
%CHECKS with data for n=8,9,10,11,12: 6-30-2018
\end{enumerate}

\end{lemma}
\begin{proof}
Using the Murnaghan-Nakayama rule, we list the different configurations for the irreducibles appearing in the expansion of $p_2^2p_{n-4}.$
The list below is organised by considering border strips of size $(n-4)$ attached to the shapes appearing in $p_2^2:$
\vskip.1in
\begin{enumerate}[label=(\arabic*)]
\item
\ytableausetup{smalltableaux}
 \ytableaushort{11XX \ldots  X, 22X, XXX,\ldots,\ldots  ,X}\qquad\qquad or \qquad\qquad
\ytableausetup{baseline}
\ytableaushort{12XX \ldots X, 12X, XXX,\ldots,\ldots, X}
\vskip.1in

Each of these contributes exactly the same set of shapes $\lambda$ containing the shape $(2,2)$.  We therefore obtain the following possibilities with (signed) multiplicity 2:
\begin{description}
\item[$\lambda_2=2$ and $\lambda_1=2$]  $(-1)^{t-1}(2,2,1^t)$ or $(-1)^{t}(2,2,2,1^t)$ 
\item[$\lambda_2=2$ and $\lambda_1\geq 3$]  $(n-2,2)$
\item[$\lambda_2=3$ and $\lambda_3=3$]  $(-1)^{t+2}(*, 3,3,1^t) $
\item[$\lambda_2=3$ and $\lambda_3=0$]  $(-1)(*,3)$
\end{description}
\item 
\vskip.2in
\ytableausetup{smalltableaux}
 \ytableaushort{1122XX \ldots X,XXXXX, X, \ldots,X,}
 \vskip.1in

\begin{comment}
$${\small\begin{young} 1 &1 &2 & 2&X &X &\ldots &X\\                           
                                X &X &X&X&X\\
                                X\\
                                \ldots\\
                                X
                                \end{young}}$$
\end{comment}
\begin{description}
\item[$\lambda\not\supset (2,2)$]  Then $\lambda$ is a hook, so $(n)$ or $(-1)^{n-1}(4,1^{n-4})$.
\item[$\lambda\supset (2,2)$]  If $\lambda_2=4$ we obtain 
$(-1)^t(4,4,1^t)$, $(-1)^t(4,3,1^t)$, $(-1)^t(4,2,1^t)$; 
finally  if $\lambda_2=5$ we obtain $(-1)^{t+1}(*,5,1^t).$
\end{description}

\item 
\vskip.2in
\ytableaushort{1 XX \ldots X, 1 X, 2 X, 2X, X,\ldots, X}
\vskip.1in
\begin{comment}
$$\begin{young} 1 &X&X & \ldots &X\\                           
                                1 &X \\
                                2&X\\
                                2 &X\\
                                X &X\\
                                \ldots\\
                                X
                                \end{young}$$ 
\end{comment}
\begin{description}
\item[$\lambda\not\supset (2,2)$]  Then $\lambda$ is a hook, so $(n-3,1^3)$ or $(-1)^{n-5}(1^{n})$.  
                 \item[$\lambda\supset (2,2)$] 
                 The possibilities are:
                 $(-1)(*,2,1^2), (-1)^2(*,2^2,1), (-1)^3(*,2^3)$ and $(-1)^t (*, 2^4, 1^t).$            
\end{description}
\item 
\vskip.2in
\ytableaushort{122X \ldots X, 1 XXX, XX, X, \ldots, X}
\vskip .1in
\begin{comment}
$$ \begin{young}
            1 & 2 & 2 & X &\ldots &X\\
            1 & X & X  &X\\
            X & X \\
            X \\
            \ldots\\
            X\\
          \end{young}$$
\end{comment}
 \begin{description}
\item[$\lambda\not\supset (2,2)$]  Then $\lambda_2=1$, so $\lambda$ is a hook,  $(-1)\cdot (-1)^{n-3}(3,1^{n-3})$ or $(-1)\cdot (n-1,1)$.  
                 \item[$\lambda\supset (2,2)$] Then $2\leq \lambda_2\leq 4.$   If $\lambda_2=2$ then 
                 $(-1)\cdot (-1)^{t+1} (3,2^2,1^t)$; 
                 if $\lambda_2=3$ then $(-1)\cdot (-1)^{t+1} (3,3,2,1^t)$;  if $\lambda_2=4$ then $(-1)\cdot (-1)^{t+2} (*,4,2,1^t)$ or $(-1)\cdot (-1) (n-4,4).$
\end{description}
\item 
\vskip .2in
\ytableaushort{11 X X \ldots X, 2 XX, 2X , X, \ldots, X}
\vskip.1in
\begin{comment}
$$ \begin{young}
            1 & 1 & X & X &\ldots &X\\
            2 & X & X  \\
            2 & X \\
            X \\
            \ldots\\
            X\\
          \end{young}$$
\end{comment}          
 \begin{description}
\item[$\lambda\not\supset (2,2)$]  Thus $\lambda_2=1,$ yielding $(-1)\cdot (n-2,1^2)$ or $(-1)\cdot (-1)^{n-5} (2,1,1,1^{n-4}).$
                 \item[$\lambda\supset (2,2)$] If $\lambda_2=2$ we have $(-1)\cdot (-1)^{n-6} (2,2,2,2,1^{n-8}).$
                 
                 If $\lambda_2=3$ we have 
                 $(-1)\cdot (-1) (*,3,1), (-1)\cdot (-1)^2 (*,3,2), $ and 
                 $(-1)\cdot (-1)^{t+3} (*,3,2^2,1^t).$                       
                                \end{description}
                                
\end{enumerate}
\end{proof}

Using Lemmas~\ref{lem2.13} and~\ref{lem2.14}, we can compute  partial sums in the reverse lexicographic order, starting from the top element $(n).$ 
\begin{proposition}\label{prop2.15}  The partial sums beginning with the top element $(n)$ are:
\begin{enumerate}[label=(\arabic*)]%[itemsep=0.5pt]
\item $p_n=(n)+(-1)^{n-1}(1^n)-(n-1,1)+(-1)^{n-2} (2,1^{n-2})+\sum_{r=2}^{n-3} (-1)^r (*, 1^{r})$ for $n\geq 2$
\item 
$\sum_{(n)\geq \mu\geq (n-1,1)} p_\mu=p_n+p_{n-1} p_1= 2(n) +\sum_{r=1}^{n-2}(-1)^r (*, 1^r)
+\sum_{r=0}^{n-4}(-1)^{r-1} (*, 2, 1^r). $
\item $\sum_{(n)\geq \mu\geq (n-2,2)} p_\mu=p_n+p_{n-1}p_1+p_{n-2}p_2$
\vskip.07in
%\begin{align*}&=
$=3(n)-2(n-1,1)+(-1)^{n} (1^n)
+\sum_{r=2}^{n-3} (-1)^r (*, 1^{r})
+\sum_{r=1}^{n-5} (-1)^{r-1} (*, 2,1^r)$
\vskip.07in
 $+\sum_{t=0}^{n-6} (-1)^{t+1}(*, 3, 1^{t})
+\sum_{t=0}^{n-6} (-1)^{t+1}(*, 2,2,1^{t})$
%\end{align*}
\vskip.07in
\item $\sum_{(n)\geq \mu\geq (n-2,1^2)} p_\mu=p_n+p_{n-1}p_1+p_{n-2}p_2+ p_{n-2} p_1^2$ ($n\geq 6$)
\vskip.07in
$=4(n) +(-1)^{n-1}(2, 1^{n-2}) +(-1)^n (2^2, 1^{n-4})+\sum_{r=1}^{n-3} (-1)^r (*, 1^r) $
\vskip.07in
$ + \sum_{r=0}^{n-5} (-1)^{r-1} (*, 2, 1^r) 
+2\cdot \sum_{r=0}^{n-6} (-1)^{r+1} (*, 3,1^r)$
\vskip.07in
\item $\sum_{(n)\geq \mu\geq (n-3,3)} p_\mu=p_n+p_{n-1}p_1+p_{n-2}p_2+ p_{n-2} p_1^2+p_{n-3}p_3$
\begin{align*}&=5(n)-2(n-1,1)-(n-2,2)+2(n-2,1^2) -(n-3,3)+(n-4,2^2)\\
&+(-1)^{n}(1^n)+2(-1)^{n-1} (2,1^{n-2})+(-1)^n (2^2,1^{n-4})+(-1)^{n}(2^3, 1^{n-6}) 
 +(-1)^{n-1}(3^2,1^{n-6})\\
&+\sum_{r={\bf 3}}^{n-4} (-1)^r(*, 1^r)+\sum_{r={\bf 2}}^{\bf n-6} (-1)^{r-1} (*,2,1^r) +2\sum_{r={\bf 1}}^{n-7} (-1)^{r+1} (*, 3,1^r)+{\bf 2\delta_{n=6}(*,3)}\\
&+\sum_{t=0}^{n-8} (-1)^{t+1}(*, 4,1^t)+\sum_{t=0}^{n-8}(-1)^{t+1}(*, 3,2,1^t)+\sum_{t=0}^{n-8}(-1)^{t+1}(*, 2^3,1^t)
\end{align*}
\item $\sum_{(n)\geq \mu\geq (n-3,2,1)} p_\mu$
\begin{align*}
&=6(n) -2(n-1,1)-(n-2,2) +(n-2,1^2)-(n-3,3)+(n-3,2,1)\\
&+2(-1)^{n-1} (2,1^{n-2}) +(-1)^n (2^2,1^{n-4})+(-1)^{n}(2^3, 1^{n-6})\\
&+(-1)^n (3,1^{n-3})+(-1)^{n-1}(3,2,1^{n-5}) \\
&+\sum_{r=3}^{n-4} (-1)^r(*, 1^r)+\sum_{r=2}^{n-6}(-1)^{r-1}(*,2,1^r) +2\sum_{r=1}^{n-7} (-1)^{r+1} (*, 3,1^r)+{\bf 2\delta_{n=6}(*,3)}\\
&+2\sum_{t=0}^{n-8} (-1)^{t-1}(*,4,1^t)+\sum_{t=0}^{n-8} (-1)^{t+1} (*, 3,2,1^t)
\end{align*}
\item $\sum_{(n)\geq \mu\geq (n-3,1^3)} p_\mu$  
%CORRECTED 2018-4-8
\begin{align*}
&=7(n)-(n-2,2)-3(n-3,3)+2(n-2,1^2)+(n-4,2^2)\\
&+(-1)^{n}(1^n)+(-1)^n (2^2,1^{n-4})+(-1)^{n-1}(2^3, 1^{n-6}) \\
&+2(-1)^{n-1} (3,2,1^{n-5}) +2(-1)^n (3,1^{n-3})+(-1)^n (3^2,1^{n-6})\\
&+\sum_{r=3}^{n-4} (-1)^r(*, 1^r)+\sum_{r=2}^{n-6} (-1)^{r-1}(*,2,1^r) +2\sum_{r=1}^{n-7} (-1)^{r+1} (*, 3,1^r)\\
&+3\sum_{t=0}^{n-8} (-1)^{t+1}(*, 4,1^t)+\sum_{t=0}^{n-8}(-1)^{t}(*, 3,2,1^t)+\sum_{t=0}^{n-8}(-1)^{t+1}(*, 2^3,1^t).
\end{align*}
%
\begin{comment}
CHECKED!
\item $p_{n-3} h_2 p_1$  REMOVE THIS AFTER CHECKING ONE MORE TIME
\begin{align*}&= (n)+(n-1,1)-(n-3,3) +(-1)^{n-2}(2,1^{n-2})+(-1)^{n-1}(2^3,1^{n-6})\\
& +(-1)^{n-1} (3,2,1^{n-5}) +(-1)^n (3,1^{n-3}) +(-1)^{n-6}(3^2,1^{n-6})\\
&+\sum_{r=0}^{n-8}(-1)^{r+1} (*, 4, 1^r)+\sum_{t=0}^{n-8} (-1)^t (*,3,2,1^t)
\end{align*}
\end{comment}
%
\item $\sum_{(n)\geq \mu\geq (n-4,4)} p_\mu$ 
%checked for n=8,9,10
\begin{align*}&=8(n)-(n-1,1)-(n-2,2)+3(n-2,1^2)
-3(n-3,3)-2(n-3,1^3)\\
&-2(n-4,4)+(n-4,3,1)+(n-4,2^2)+2(n-5,3,2)-(n-5,2^2,1)\\
&+(-1)^n 2(1^n)+(-1)^{n-1} (2,1^{n-2}) +(-1)^n (2^2,1^{n-4}) 
+(-1)^{n-1}(2^3,1^{n-6}) +(-1)^n(2^4,1^{n-8})\\
&+2(-1)^{n-1} (3,2,1^{n-5}) +3(-1)^n (3,1^{n-3}) +(-1)^n (3^2,1^{n-6}) +(-1)^n(3^2,2,1^{n-8})\\
& +(-1)^{n-1}(3,2^2,1^{n-7}) +(-1)^{n-1}(4,3,1^{n-7})
+(-1)^n (4^2,1^{n-8})\\
%&+(-1)^{n-1}(4,1^{n-4})+\sum_{r={\bf 4}}^{n-4} (-1)^r (*,1^r)\\remove last term (r=n-4) of sum in next line
%& +(-1)^n (4,2,1^{n-6})+\sum_{r={\bf 3}}^{n-6} (-1)^{r-1}(*,2,1^r)\\remove last term (r=n-6) of sum in next line
&+\sum_{r={\bf 4}}^{\bf{n-5}} (-1)^r (*,1^r)
+\sum_{r={\bf 3}}^{\bf{n-7}} (-1)^{r-1}(*,2,1^r)\\
& +2\sum_{r={\bf 2}}^{n-7} (-1)^{r+1} (*, 3,1^r)
+3\sum_{t={\bf 1}}^{n-8} (-1)^{t+1}(*, 4,1^t)+\sum_{t={\bf 1}}^{n-8}(-1)^{t}(*, 3,2,1^t)+\sum_{t={\bf 1}}^{n-8}(-1)^{t+1}(*, 2^3,1^t)\\
&+\sum_{t=0}^{n-10}(-1)^{t+1}\left[(*,5,1^t)+(*,4,2,1^t)+(*,3,2^2,1^t)+(*,2^4,1^t)\right]
\end{align*}
%TEMPORARY: REMOVED AFTER ADDING TO (8)
%%%%%%%%%%%%%%%%%%%%%%%%
\begin{comment}
\item $p_{n-4} p_3 p_1$ %checked for n=8 to 12
\begin{align*}&=(n)-(n-2,2)+(n-3,3)+(n-3,1^3)-(n-4,2,1^2)+(n-5,2^2,1)\\
&+(-1)^{n-1}[(1^n)-(2^2,1^{n-4})+(2^3,1^{n-6})+(4,3,1^{n-7})-(4,2,1^{n-6})+(4,1^{n-4})]\\
&+(-1)^n(4^2,1^{n-8})
-(n-6,2^3)\\
&+\sum_{t=0}^{n-10} (-1)^{t-1}(*,5,1^t)+
\sum_{t=0}^{n-9} (-1)^{t-1}(*,3^2,1^t)+\sum_{t=1}^{n-9} (-1)^{t-1}(*,2^4,1^{t-1}) 
\end{align*}
\end{comment}
%%%%%%%%%%%%%%%%%%%%%%%%
\item $\sum_{(n)\geq \mu\geq (n-4,3,1)} p_\mu$
\begin{align*}&=9(n)-(n-1,1)-2(n-2,2)+3(n-2,1^2)
-2(n-3,3)-(n-3,1^3)\\[ 3pt]
&-2(n-4,4)+(n-4,3,1)+(n-4,2^2)-(n-4,2,1^2)+2(n-5,3,2)-(n-6,2^3)\\[ 3pt]
&+(-1)^n (1^n)+(-1)^{n-1} (2,1^{n-2}) +2(-1)^n (2^2,1^{n-4}) 
+2(-1)^{n-1}(2^3,1^{n-6})\\
& +(-1)^n(2^4,1^{n-8})\\
&+2(-1)^{n-1} (3,2,1^{n-5}) +3(-1)^n (3,1^{n-3}) +(-1)^n (3^2,1^{n-6}) +(-1)^n(3^2,2,1^{n-8}) \\[ 3pt]
&+(-1)^{n-1}(3,2^2,1^{n-7})\\[ 3pt]
& +2(-1)^{n-1}(4,3,1^{n-7})
+2(-1)^n (4^2,1^{n-8})+(-1)^n(4,2,1^{n-6})+(-1)^{n-1}(4,1^{n-4})\\[ 3pt]
%%%
&+\sum_{r={\bf 4}}^{\bf{n-5}} (-1)^r (*,1^r)
+\sum_{r={\bf 3}}^{\bf{n-7}} (-1)^{r-1}(*,2,1^r)\\[ 3pt]
& +2\sum_{r={\bf 2}}^{n-7} (-1)^{r+1} (*, 3,1^r)
+3\sum_{t={\bf 1}}^{n-8} (-1)^{t+1}(*, 4,1^t)+\sum_{t={\bf 1}}^{n-8}(-1)^{t}(*, 3,2,1^t)\\
&+\sum_{t={\bf 1}}^{n-8}(-1)^{t+1}(*, 2^3,1^t)\\[ 3pt]
&+\sum_{t=0}^{n-10}(-1)^{t+1}\left[2(*,5,1^t)+(*,4,2,1^t)+(*,3,2^2,1^t)\right] +\sum_{t=0}^{n-9} (-1)^{t-1}(*,3^2,1^t)
\end{align*}
% TEMPORARY: removed after adding to above so that I can add to above
%%%%%%%%%%%%%%%%%%%%%%%%%
\begin{comment}
\item $p_{n-4}p_2^2$  
%CHECKS with data for n=8,9,10,11,12
\begin{align*}&=(n)-(n-1,1) +2 (n-2,2) -(n-2,1^2) -2(n-3,3) +(n-3,1^3)\\
&+ (n-4,4)+(n-4,3,1)-(n-4,2,1^2) -(n-5,3,2)
+(n-5,2^2,1)-(n-6,2^3)\\
&+(-1)^{n-1}[(1^n)-(2,1^{n-2})+2(2^2,1^{n-4}) 
-2(2^3,1^{n-6}) +(2^4,1^{n-8})]\\
&+(-1)^n (3,1^{n-3}) +(-1)^{n-1} (3,2^2,1^{n-7}) +(-1)^n (3^2,2,1^{n-8})\\
&+(-1)^{n}(4^2, 1^{n-8})+(-1)^{n-1}(4,3,1^{n-7}) +(-1)^{n} (4,2,1^{n-6})+(-1)^{n-1} (4,1^{n-4})\\
&+\sum_{t=0}^{n-10} (-1)^t [(*,2^4,1^t)+(*,3,2^2,1^t)]
+\sum_{t=0}^{n-10} (-1)^{t-1} [(*,5,1^t)+(*,4,2,1^t)]\\
&+2\sum_{t=0}^{n-9} (-1)^t (*,3^2,1^t)
\end{align*}
\end{comment}
%above is temporary
%%%%%%%%%%%%%%%%%%%%%%%%%
\item  $\sum_{(n)\geq \mu\geq (n-4,2^2)} p_\mu$
%checked for n=10
\begin{align*}&=10(n)-2(n-1,1)+2(n-2,1^2)-{\color{red}4}(n-3,3)\\[ 3pt]
&-(n-4,4)+2(n-4,3,1)+(n-4,2^2)-2(n-4,2,1^2)\\[ 3pt]
&+(n-5,3,2)+(n-5,2^2,1)-2(n-6,2^3)+(-1)^n{\color{red}4}(3,1^{n-3})\\[ 3pt]
&+2(-1)^{n-1}(3,2,1^{n-5})+(-1)^n(3^2,1^{n-6})+{\bf 3}(-1)^n(3^2,2,1^{n-8})+2(-1)^{n-1}(3,2^2,1^{n-7})\\[ 3pt]
&+(-1)^{n-1}(4,3,1^{n-7})+{\bf 0}(-1)^n(4^2,1^{n-8}) +(-1)^n{\bf 2}(4,2,1^{n-6}) +2(-1)^{n-1}(4,1^{n-4})\\[ 3pt]
%%%
&+\sum_{r={\bf 4}}^{\bf{n-5}} (-1)^r (*,1^r)
+\sum_{r={\bf 3}}^{\bf{n-7}} (-1)^{r-1}(*,2,1^r)
 +2\sum_{r={\bf 2}}^{\bf n-8} (-1)^{r+1} (*, 3,1^r)\\
&+3\sum_{t={\bf 1}}^{\bf n-9} (-1)^{t+1}(*, 4,1^t)\\
&+\sum_{t={\bf 1}}^{\bf n-9}(-1)^{t}(*, 3,2,1^t)+\sum_{t={\bf 1}}^{n-8}(-1)^{t+1}(*, 2^3,1^t)+\sum_{t=0}^{n-10} (-1)^t (*,2^4,1^t)\\[ 3pt]
&+\sum_{t=0}^{n-10}(-1)^{t-1}\left[{\bf 3}(*,5,1^t)+{\bf 2}(*,4,2,1^t)\right] +\sum_{t=0}^{n-9} (-1)^{t}(*,3^2,1^t)
\end{align*}
%%following is temporary, to make it easy to add to above
%%% now removed%%%%%%%%%%%%%%%%%
\begin{comment}
\item $p_{n-4} p_2 p_1^2$ %checked for n=8 to 11
\begin{align*}&=(n)+(n-1,1)-(n-2,1^2)-(n-3,1^3) -(n-4,4)\delta_{n\ge 9} +(n-4,3,1)\\ 
& +(n-4,2,1^2)-(n-5,3,2) -(n-5,2^2,1)+(n-6,2^3) \\
& +(-1)^n[(1^n)+(2,1^{n-2})-(2^4,1^{n-8}) -(3,1^{n-3}) +(3,2^2,1^{n-7})-(3^2,2,1^{n-8}) ]\\
& +(-1)^n [(-(4,1^{n-4})+(4,2,1^{n-6})-(4,3,1^{n-7}) +(4^2,1^{n-8}) ]\\
& +\sum_{t=0}^{n-10} (-1)^{t-1} \{(*,5,1^t)-(*,3,2^2,1^t)\} 
 +\sum_{r=0}^{n-10} (-1)^{r}\{(*,4,2,1^{r}) -(*,2^4,1^{r})\}
 \end{align*}
\end{comment}
%%%%%%%%%%%%%%
\item $\sum_{(n)\geq \mu\geq (n-4,2,1^2)} p_\mu$
%checked for n=10!  6-30-2018
\begin{align*}
&=11(n) -(n-1,1)+(n-2,1^2)-{\color{red}4}(n-3,3)-(n-3,1^3)\\
&-(1+\delta_{n\geq 9})(n-4,4)+3(n-4,3,1)+(n-4,2^2)-(n-4,2,1^2)\\
&-(n-6,2^3) +(-1)^n (1^n)+(-1)^n (2,1^{n-2}) +{\bf 2}(-1)^{n-1} (2^4, 1^{n-8})\\
&+(-1)^n{\color{red}3}(3,1^{n-3})+2(-1)^{n-1}(3,2,1^{n-5})+(-1)^n(3^2,1^{n-6})+{\bf 2}(-1)^n(3^2,2,1^{n-8})\\
&+(-1)^{n-1}(3,2^2,1^{n-7}) \\
&+{\bf 2}(-1)^{n-1}(4,3,1^{n-7})+(-1)^n(4^2,1^{n-8}) +(-1)^n{\bf 3}(4,2,1^{n-6}) +(-1)^{n-1}{\bf 3}(4,1^{n-4})\\
%%%
&+\sum_{r={\bf 4}}^{\bf{n-5}} (-1)^r (*,1^r)
+\sum_{r={\bf 3}}^{\bf{n-7}} (-1)^{r-1}(*,2,1^r)
 +2\sum_{r={\bf 2}}^{\bf n-8} (-1)^{r+1} (*, 3,1^r)
+3\sum_{t={\bf 1}}^{\bf n-9} (-1)^{t+1}(*, 4,1^t)\\
&+\sum_{t={\bf 1}}^{\bf n-9}(-1)^{t}(*, 3,2,1^t)+\sum_{t={\bf 1}}^{\bf n-9}(-1)^{t+1}(*, 2^3,1^t)\\
&+\sum_{t=0}^{n-10}(-1)^{t-1}\left[{\color{red} 4}(*,5,1^t)+{\bf 1}(*,4,2,1^t)-(*,3,2^2,1^t)\right] +\sum_{t=0}^{n-9} (-1)^{t}(*,3^2,1^t)
\end{align*}

\end{enumerate}

\end{proposition}

\begin{proof}  We sketch the proof, since the details are routine and tedious.  In general, each partial sum is obtained from the preceding one by adding the expansion of the power sum indexed by the appropriate partition.  More precisely, if 
$\mu^+$ covers $\mu$ in reverse lexicographic order, then 
$\psi_{\mu^+}=\psi_\mu+p_{\mu^+}.$

Thus the first two partial sums follow by adding the first two sums in Lemma~\ref{lem2.13}. A similar procedure is applied for the remaining sums, with the following exceptions.
 
 For (3), we compute the partial sum by using 
$p_2+p_1^2=2h_2$ and thus it suffices to add the expansion of $2h_2p_{n-2}$ from Lemma~\ref{lem2.13} to the preceding partial sum.

Similarly for (7), we compute the sum 
$\sum_{(n)\geq \mu\geq (n-3,1^3)} p_\mu$ 
 by adding $p_{n-3} (p_2p_1+p_1^3)=2 p_{n-3} h_2 p_1$ (using the expansion  (8) of Lemma~\ref{lem2.13}),   to 
$\sum_{(n)\geq \mu\geq (n-3,3)} p_\mu,$ which is given in (5) above. 

In general, in all cases we use the relevant computations of Lemma~\ref{lem2.13}, the chief exception being the expression in (10), which requires the expansion of $p_{n-4} p_2^2$ computed in Lemma~\ref{lem2.14}.  The expression (11) of this proposition then follows cumulatively using {\bf Part (11)} of Lemma~\ref{lem2.13}.

It is important to note that the sums have been carefully rewritten so that there is no \lq\lq collapsing'': as an example, we give here an analysis of what happens in computing the partial sum (10).  This sum  is obtained by adding to the sum in (9) the expansion for $p_{n-4}p_2^2,$ which from Lemma~\ref{lem2.14} is 
%\begin{comment}
%\item $p_{n-4}p_2^2$  %CHECKS with data for n=8,9,10,11,12
%
\begin{align*} &{\scriptstyle  (n)-(n-1,1) +2 (n-2,2) -(n-2,1^2) -2(n-3,3) +(n-3,1^3)}\\
&{\scriptstyle + (n-4,4)+(n-4,3,1)-(n-4,2,1^2) -(n-5,3,2)
+(n-5,2^2,1)-(n-6,2^3)}\\
&{\scriptstyle +(-1)^{n-1}[(1^n)-(2,1^{n-2})+2(2^2,1^{n-4}) 
-2(2^3,1^{n-6}) +(2^4,1^{n-8})]}\\
&{\scriptstyle +(-1)^n (3,1^{n-3}) +(-1)^{n-1} (3,2^2,1^{n-7}) +(-1)^n (3^2,2,1^{n-8})}\\
	&{\scriptstyle +(-1)^{n}(4^2, 1^{n-8})+(-1)^{n-1}(4,3,1^{n-7}) +(-1)^{n} (4,2,1^{n-6})+(-1)^{n-1} (4,1^{n-4})}\\
&{\scriptstyle +\sum_{t=0}^{n-10} (-1)^t [(*,2^4,1^t)+(*,3,2^2,1^t)]
+\sum_{t=0}^{n-10} (-1)^{t-1} [(*,5,1^t)+(*,4,2,1^t)]}\\
&{\scriptstyle +2\sum_{t=0}^{n-9} (-1)^t (*,3^2,1^t)  }
\end{align*}
%\end{comment}
Adding this to (9) of Proposition~\ref{prop2.15} produces 

 $\sum_{(n)\geq \mu\geq (n-4,2^2)} p_\mu$
%checked for n=10
\begin{align*}&{\scriptstyle =10(n)-2(n-1,1)+2(n-2,1^2)-{\color{red}4}(n-3,3)}\\
&{\scriptstyle -(n-4,4)+2(n-4,3,1)+(n-4,2^2)-2(n-4,2,1^2)}\\
&{\scriptstyle +(n-5,3,2)+(n-5,2^2,1)-2(n-6,2^3)+(-1)^n{\color{red}4}(3,1^{n-3})}\\
&{\scriptstyle +2(-1)^{n-1}(3,2,1^{n-5})+(-1)^n(3^2,1^{n-6})+\textcolor{ magenta}{(A1)\ 2(-1)^n(3^2,2,1^{n-8})}+2(-1)^{n-1}(3,2^2,1^{n-7})}\\
&{\scriptstyle +\textcolor{ blue}{(A2)\ {\bf 3}(-1)^{n-1}(4,3,1^{n-7})}+\textcolor{ cyan}{(A3)\ \bf 3(-1)^n(4^2,1^{n-8})} +(-1)^n{\bf 2}(4,2,1^{n-6}) +2(-1)^{n-1}(4,1^{n-4})}\\
%%%
&{\scriptstyle +\sum_{r={\bf 4}}^{\bf{n-5}} (-1)^r (*,1^r)
+\sum_{r={\bf 3}}^{\bf{n-7}} (-1)^{r-1}(*,2,1^r)
 +\textcolor{ blue}{(A2)\ 2\sum_{r={\bf 2}}^{n-7} (-1)^{r+1} (*, 3,1^r)}
+\textcolor{ cyan}{(A3)\ \bf 3\sum_{t={\bf 1}}^{n-8} (-1)^{t+1}(*, 4,1^t)}}\\
&{\scriptstyle +\textcolor{ magenta}{(A1)\ \sum_{t={\bf 1}}^{n-8}(-1)^{t}(*, 3,2,1^t)}+\sum_{t={\bf 1}}^{n-8}(-1)^{t+1}(*, 2^3,1^t)+\sum_{t=0}^{n-10} (-1)^t (*,2^4,1^t)}\\
&{\scriptstyle +\sum_{t=0}^{n-10}(-1)^{t-1}\left[{\bf 3}(*,5,1^t)+{\bf 2}(*,4,2,1^t)\right] +\sum_{t=0}^{n-9} (-1)^{t}(*,3^2,1^t)
}
\end{align*}

The items of matching colour (or matching labels, in the absence of colour) can be combined as follows:

$ \qquad{\scriptstyle \textcolor{blue}{(A2)\ {\bf 3}(-1)^{n-1}(4,3,1^{n-7})}+\textcolor{blue}{2\sum_{r={\bf 2}}^{n-7} (-1)^{r+1} (*, 3,1^r)}=\textcolor{blue}{(-1)^{n-1}(4,3,1^{n-7})+2\sum_{r={\bf 2}}^{n-8} (-1)^{r+1} (*, 3,1^r) }} ;$

$\qquad{\scriptstyle\textcolor{ cyan}{\bf (A3)\  3(-1)^n(4^2,1^{n-8})}+
\textcolor{cyan}{\bf 3\sum_{t={\bf 1}}^{n-8} (-1)^{t+1}(*, 4,1^t)}
=\textcolor{cyan}{\bf 3\sum_{t={\bf 1}}^{n-9} (-1)^{t+1}(*, 4,1^t)}};$

$\qquad{\scriptstyle   \textcolor{ magenta}{(A1)\  2(-1)^n(3^2,2,1^{n-8})}   +\textcolor{magenta}{\sum_{t={\bf 1}}^{n-8}(-1)^{t}(*, 3,2,1^t)}   =\textcolor{magenta}{3(-1)^n(3^2,2,1^{n-8})} +\textcolor{magenta}{\sum_{t={\bf 1}}^{n-9}(-1)^{t}(*, 3,2,1^t)} } $

Making these replacements finally yields the completely reduced expression (10).

Likewise, the reduced expression (11) is obtained by adding to (10) the expansion of $p_{n-4}p_2p_1^2.$ There is only one pair that recombines into one term here, namely 

\qquad${\scriptstyle (-1)^{n-1} (2^4, 1^{n-8})+\sum_{t=1}^{n-8}(-1)^{t+1} (*,2^3,1^t)=2(-1)^{n-1} (2^4, 1^{n-8})+\sum_{t=1}^{n-9}(-1)^{t+1} (*,2^3,1^t)}.$

\end{proof}

We can now deduce the positivity of  the functions $\psi_\mu$ for $\mu\geq (n-4,1^4).$

% Don't think I need this? (6-30-2018)
%\begin{lemma} One has the recurrences
% $$ \psi_{(n-r,r)}=\psi_{(n-r+1, r-1)}-p_{n-r+1}\psi_{r-1},$$ 
% and, for any partition $\mu$ of $r,$ 
% $$\psi_{(n-r,\mu)}= \psi_{(n-r+1, r-1)} -p_{n-r+1} \psi_{r-1} -p_{n-r} (\psi_r-\psi_\mu).$$
% In particular, for all $\mu\in [(1^n), (n)),$ one has 
% $$\psi_{\mu}=\psi_n-\sum_{\nu\vdash n:\mu<\nu\leq (n)} p_\nu.$$
%\end{lemma}

\begin{theorem}\label{thm2.16}  Let $n\geq 6.$  Let $\mu$ be a partition in the interval $[(n-4,1^{n-4}), (n)]$ in the reverse lexicographic order on partitions of $n.$ Then $\psi_\mu$ is Schur-positive. 
%Moreover, the multiplicity of the irreducible indexed by $\lambda$ in $\psi_{(n-1,1)}$ differs from the multiplicity in $\psi_n$ if and only if $\lambda$ is a hook, and in that case it differs by exactly one.

\end{theorem}
\begin{proof} 
It is clear from the definition that 
 \begin{equation}\psi_{\mu}=\psi_n-\sum_{\nu\vdash n:\mu<\nu\leq (n)} p_\nu.\end{equation}
 We use the partial sum computations in Proposition~\ref{prop2.15}. 
Observe that, in each of those expansions, no Schur function appears with multiplicity greater than $+4$, except for the trivial representation, which appears with multiplicity equal to the number of partitions in the interval $(\mu, (n)].$  

For example, for $\psi_{(n-4,3,1)}$ we see from (8) that we need to subtract from $\psi_n$ a virtual representation in which no multiplicity in the sum exceeds $+3$, other than the multiplicity of the trivial representation which is now 8. 

Similarly to obtain $\psi_{(n-4,2^2)},$ from (9) it follows that we subtract from $\psi_n$ a representation in which no multiplicity in the sum exceeds $+3$, other than the multiplicity of the trivial representation which is now 9.

In fact the largest  multiplicity (in absolute value) of $+4$ is obtained for the first time in the penultimate sum (10) of Proposition~\ref{prop2.15}, $\sum_{(n)\geq \mu\geq (n-4,2^2)} p_\mu$  (for the two irreducibles $(n-3,3), (3,1^{n-3})$).   In the last sum of Proposition~\ref{prop2.15}, viz.
$\sum_{(n)\geq \mu\geq (n-4,2,1^2)} p_\mu,$ again the largest multiplicity in absolute value is 4, and this multiplicity occurs several times.

%It is clear from the definition that 
% $$\psi_{\mu}=\psi_n-\sum_{\nu\vdash n:\mu<\nu\leq (n)} p_\nu.$$
 
 But Lemma~\ref{lem2.6} guarantees that $\psi_n$ has multiplicity at least 4 for each irreducible except the trivial module, and hence, 
examining the partial sums in Proposition~\ref{prop2.15}, it is clear that the right-hand side is Schur-positive in all the cases enumerated.  The fact that all the expressions are reduced (no further simplification occurs) is important in this argument.  The multiplicity of the trivial representation is the partition number $p(n),$ which is certainly at least the length of the interval $(\mu, (n)].$  The theorem is proved.  \end{proof}

Together Theorems~\ref{thm2.11} and~\ref{thm2.16} complete the proof of Theorem~\ref{thm1.3}.  

It is difficult to see how to generalise this argument.  Already for $n=8,9,10,$ computation with Maple shows that in the Schur function expansion of the sum $\sum_{\lambda\geq (n-4,1^4)} p_\lambda, $  $s_{(n-3,3)}$ occurs with multiplicity  $-6;$  $s_{(n-4,4)}$ occurs with multiplicity $-5$ when $n=9,10,$ and $s_{(n-4,3,1)}$ occurs with multiplicity $-5$ for $n=8.$  The lower bound that we were able to establish in Lemma~\ref{lem2.6} is therefore insufficient to guarantee Schur positivity of $\psi_\mu$ by these arguments, in the case when $\mu$ is strictly below $(n-4,1^4)$ in reverse lexicographic order.

From Theorem~\ref{thm2.11} and Proposition~\ref{prop2.4} it is also easy to derive the following information about the multiplicity of the sign representation.  Clearly if $\mu$ and $\nu$ are consecutive partitions in reverse lexicographic order, this multiplicity differs by 1 in absolute value, since $\psi_\mu-\psi_\nu=\pm p_\nu$.  It is also clear that the multiplicity of the trivial representation decreases by one as we descend the chain from $(n)$ to $(1^n).$

Denote by $\langle, \rangle$  the inner product on the ring of symmetric functions for which the Schur functions form an orthonormal basis.
\begin{corollary}\label{cor2.17}  We have 
\begin{enumerate}[label=(\arabic*)]
\item $\langle \psi_n, s_{(1^n)}\rangle= \langle \psi_{(n-2,2)}, s_{(1^n)}\rangle=$ the number of partitions of self-conjugate partitions of $n.$
\item $\langle \psi_{(n-1,1)}, s_{(1^n)}\rangle=\langle \psi_n, s_{(1^n)}\rangle +(-1)^n;$
\item $\langle \psi_{(n-2,1^2)}, s_{(1^n)}\rangle=\langle \psi_n, s_{(1^n)}\rangle -(-1)^n;$
\item $\langle \psi_{(n-3,3)}, s_{(1^n)}\rangle=\langle \psi_n, s_{(1^n)}\rangle.$
\item $\langle \psi_{(2^k, 1^{n-2k})}, s_{(1^n)}\rangle=(-1)^k +
\langle \psi_{(2^{k-1}, 1^{n-2k+2})}, s_{(1^n)}\rangle 
,\quad   1\leq k\leq \lfloor\frac{n}{ 2} \rfloor .$
\item $\langle \psi_{(3, 1^{n-3})}, s_{(1^n)}\rangle =
1+ \langle \psi_{(2^k, 1^{n-2k})}, s_{(1^n)}\rangle$, 
where  $k= \lfloor\frac{n}{ 2} \rfloor$
\end{enumerate}
\end{corollary}

  In Example~\ref{ex1.4}, we have underlined  the increasing runs, and italicised  the decreasing runs, of length 3 or more, in the multiplicity $\langle \psi_\mu, s_{(1^n)}\rangle.$ It is unclear how to predict the runs of 1's and $(-1)$'s, i.e. the increasing and decreasing sequences in the partial sums. The longest such run in the examples occurs in $S_{13}$, namely $1,2,3,4.$ The corresponding partitions are 
$$  (6, 5, 2)< (6, 6, 1)<(7, 1^6), $$ all with sign $+1.$

%Runs of length 4 show up again for n=15, 20, 22.$
%%%%%%%%%%%%%%%%%%%%%%%%%
\begin{comment}  
The generating function for partitions by number of parts is 
$$\prod_i(1-tq^i)^{-1}= \sum_\lambda t^{\ell(\lambda) }q^{|\lambda|}$$
Hence $\sum_{\lambda\vdash n} t^{\ell(\lambda)}$ is the coefficient of $q^n$ in $\prod_i (1-tq^i)^{-1},$ and thus 
$\sum_{\lambda\vdash n} (-1)^{n-\ell(\lambda)}$ is the coefficient of $q^n$ in 
$$\prod_i (1- (-1) (-q)^i )^{-1}=\prod_j (1-q^{2j-1})^{-1} (1+q^{2j})^{-1}= \frac{\prod_k (1-q^{k})^{-1}}{\prod_k (1-q^{2k})^{-1}}\cdot \prod_j (1+q^{2j})^{-1}$$
$$=\prod_k (1-q^{k})^{-1}\cdot \prod_j(1-q^j)\prod_j (1+q^j) 
\cdot \prod_j (1+q^{2j})^{-1}$$
$$=\frac{\prod_j (1+q^j) }{\prod_j (1+q^{2j})}=\prod_j(1+q^{2j-1})$$
\end{comment}

%
%More generally, for $k\geq 3,$ if $T_k=\{1,\ldots,k-1\},$ we have 
 %$$ \psi_{(k, 1^{n-k})}= p_k p_1^{n-k} +p_{n,T_k}.$$ 
 %In \cite{Su3} it is conjectured that $p_{n,T_k}$ is Schur-positive for all $k\geq 3.$ 
\vskip .1in
Schur positivity also holds for the following (unsaturated) chain of partitions in reverse lexicographic order.

\begin{proposition}\label{prop2.18}
Let $T_n=\{\lambda\vdash n: \lambda=(n-r, 1^r), 0\leq r\leq n-1\}.$  Then $Hk_n=\sum_{\mu\in T_n} p_\mu$ is Schur-positive.  In fact $Hk_n$ contains all irreducibles unless $n$ is even, in which case only the irreducible indexed by $(1^n)$ does not appear.
\end{proposition}
\begin{proof} Clearly $Hk_{n}=p_{n}+p_1 Hk_{n-1}.$ 
Since $p_n=\sum_{r=0}^n (-1)^r s_{(n-r, 1^r)},$
 by Frobenius reciprocity,  we have 

$\langle Hk_n, s_\lambda\rangle=\begin{cases} \langle Hk_{n-1}, s_{\lambda/(1)}\rangle, & \text{ if }\lambda \text{ is not a hook}; \quad (A)\\
\langle  Hk_{n-1}, s_{( 1^{n-1})}\rangle + (-1)^{n-1}, & \lambda=(1^n); \quad (B)\\
\langle  Hk_{n-1}, s_{(n-1)}\rangle + 1, & \lambda=(n); \quad(C)\\
\langle  Hk_{n-1}, s_{(n-r-1, 1^r)}\rangle \\
+\langle  Hk_{n-1}, s_{(n-r, 1^{r-1})}\rangle +(-1)^r, & \lambda=(n-r, 1^r),1\leq r\leq n-2.  \quad (D)
\end{cases}$

We verify that $Hk_1=s_{(1)},$ $Hk_2=p_2+p_1^2=2 s_{(2)},$ 
$Hk_3=\psi_3=3 s_{(3)}+ s_{(2,1)}+ s_{(1^3)},$ 
$Hk_4=4 s_{(4)}+3 s_{(3,1)}+3 s_{(2, 1^2)}+ s_{(2,2)},$ 
$Hk_5=5 s_{(5)}+6 s_{(4,1)}+7 s_{(3, 1^2)}+ 2s_{(2,1^3)}+s_{(1^5)} + 4 s_{(3,2)}+ 4 s_{(2^2,1)}.$ 

First we claim that 
\begin{enumerate}[label=(\arabic*)]
\item $\langle Hk_n, s_{(1^n)}\rangle=\begin{cases} 1, &\ n \text{  odd};\\
0, &\text{otherwise}.\end{cases}$
\item $ \langle Hk_n, s_{(n)}\rangle=n$ for all n.
\item $\langle Hk_{n-1}, s_{(n-r, 1^r)}\rangle\geq 1$ for all $r.$
\end{enumerate}

Claims (1) and (2) are immediate by an easy induction from (B) and (C) above.

We will show that claim (3) also follows by induction. It is clearly true for $r=0.$ For $r=1,$ we have, using (2),  the recurrence
 \begin{center}
$\langle s_{(n-1,1)}, Hk_n\rangle-\langle s_{(n-2,1)}, Hk_{n-1}\rangle=(n-2),$ 
\end{center} and hence, since $\langle s_{(2,1)}, Hk_3\rangle =1,$ 
\begin{center}
$\langle s_{(n-1,1)}, Hk_n\rangle=\sum_{r=2}^{n-1} (n-r)={{n-1}\choose{2}}\geq 3$ if $n\geq 4.$ \end{center}
Similarly we have, for $r=2,$ 
\begin{center}
$\langle s_{(n-2,1^2)}, Hk_n\rangle-\langle s_{(n-3,1^2)}, Hk_{n-1}\rangle =(-1)^2+\langle s_{(n-2,1)}, Hk_{n-1}\rangle 
=1+{n-2\choose 2} $ if $n\geq 4.$ \end{center} 
Taking this recurrence down to the last line, namely
\begin{center}
$\langle s_{(2,1^2)}, Hk_4\rangle-\langle s_{(1,1^2)}, Hk_{3}\rangle =(-1)^2 +\langle s_{(2,1)}, Hk_3\rangle,$\end{center} 
we have, 

\begin{center}
$\langle s_{(n-2,1^2)}, Hk_n\rangle =\langle s_{(1,1^2)}, Hk_{3}\rangle+(n-3) +
\sum_{j=2}^{n-2} {j\choose 2} =n-2 +{n-1\choose 3}\geq 3, 
\text{ if }  n\geq 4.$ \end{center}

Our induction hypothesis will be that $\langle s_{(n-r,1^r)}, Hk_n\rangle \geq 3$ for some $r$ such that $r\leq n-2.$ This has now been verified  for $r=0,1,2.$  Then it follows from (D) above, using the same telescoping sum, that 

%%$Hk_n$ is Schur-positive for $n<m.$   In view of the preceding results, we need only examine the case when $\lambda$ is a hook $(n-r, 1^r)$ and $3\leq r \leq n-2.$
 
\begin{center}
$\langle s_{(n-r,1^r)}, Hk_n\rangle 
- \langle s_{(n-(r-2),1^{r-(r-2)})}, Hk_{n-r+2}\rangle \geq 
\langle s_{(n-(r-1),1^{r-(r-1)})}, Hk_{n-r+1}\rangle
 +\sum_{i=2}^r (-1)^i\geq 0$ since $n\geq r+2,$
\end{center}
and hence 
\begin{center} $\langle s_{(n-r,1^r)}, Hk_n\rangle \geq 
\langle s_{(n-(r-2),1^{2})}, Hk_{n-r+2}\rangle \geq 3$ \end{center}
by induction hypothesis.

This establishes the induction step and hence claim (3). In view of (A) above, the positivity of the multiplicities of hooks in $Hk_n$ implies  that $Hk_n$ is Schur-positive for all $n.$  The last statement is clear.
%
%\begin{equation*}
%\langle s_{(n-r,1^r)}, Hk_n\rangle=\sum_{j=0}^r 
%\langle s_{(n-r+1,1^{r-j})}, Hk_{n-j+1}\rangle +(-1)^r
%=
%\end{equation*}
%
\end{proof}

 The representations $Hk_n$ give rise to an interesting family of nonnegative integers.
 
 Let $a_{n,r}=\langle s_{(n-r, 1^r)}, Hk_n\rangle, 0\leq r\leq n-1.$ Table 1 gives the first few values of this sequence.
 
 Using $Hk_n=p_n+p_1Hk_{n-1},$ from $(D)$ in Proposition~\ref{prop2.18} we have the recurrence
 $$a_{n,r}=(-1)^r+a_{n-1,r}+a_{n-1,r-1}, 1\leq r\leq n-2, n\geq 2$$
 and $a_{n,0}=n,  a_{n,n-1}=\delta_{n \text{ odd}}$ for all $n.$
 Thus $a_{n,1}-a_{n-1,1}=n-2,$ giving $a_{n,1}={n-1\choose 2}$ 
 for $n\geq 2.$ 
 
% \begin{center}{\small Table 1: $a_{n,r},$ row $n\geq 1,$ column $r\geq 0$}\end{center}
%
\begin{center}
{\small Table 1: $a_{n,r},$ row $n\geq 1,$ column $r\geq 0$}
%\begin{tabular}{|c|c|c|c|c|c|c|c|c|c|c|}\hline
\begin{tabular}{|c|c|c|c|c|c|c|c|c|c|c|c|}\hline
1  & & &  &  & & & & & & &\\ \hline
2  & 0 & & & & & & & & & &\\ \hline
3  &1 &1 &  &  & & & & & & & \\ \hline
4 &3  &3 &0  &  & & & & & & &\\ \hline 
5&6 &7 &2  &1  & & & & & & &\\ \hline
6 &10 &14 &8  &4  &0 & & & & & &\\ \hline
7 &15 &25 &21  &13  &3 &1 & & & & &\\ \hline
8 &21 &41 &45  &35  &15 &5 &0 & & & &\\ \hline
9 &28 &63 &85  &81  &49 &21 &4 &1 & & &\\ \hline
10 &36 &92 &147  &167  &129 &71 &24 &6 &0 & &\\ \hline
11&45 &129 &238  &315  &295 &201 &94 &31 &5 &1 &  \\ \hline
12 &55 &175 &366  &554  &609 &497 &294 &126 &35 &7 &0  \\ \hline
\end{tabular}
\end{center}
 \vskip.2in
 Define $b_{n,r}=\langle s_{n-r,r}, Hk_n\rangle, 0\leq r\leq \frac{n}{2}.$   Then $(A)$ gives
 $$b_{n,r}=b_{n-1,r}+b_{n-1,r-1}, r\geq 2,\quad b_{n,1}=a_{n,1}= {n-1\choose 2}.$$ 
 
 \begin{comment}
 \begin{center}{\small Table 2: $b_{n,r},$ row $n\geq 1,$ column $r\geq 0$}\end{center}
\begin{center}
%\begin{tabular}{|c|c|c|c|c|c|c|c|c|c|c|}\hline
\begin{tabular}{|c|c|c|c|c|c|c|}\hline
1  & & &  &  & & \\ \hline
2  &  & & & & & \\ \hline
3  &1 & &  &  & &  \\ \hline
4 &3  &1 &  &  & & \\ \hline 
5&6 &4 &  &  & & \\ \hline
6 &10 &10 &4  &  & & \\ \hline
7 &15 &20 &14  &  & & \\ \hline
8 &21 &35 &34  &14  & & \\ \hline
9 &28 &56 &69  &48  & &\\ \hline
10 &36 &84 &125  &117  &48 & \\ \hline
%11&45 &129 &238  &315  &295 &201 &94 &31 &5 &1 &  \\ \hline
%12 &55 &175 &366  &554  &609 &497 &294 &126 &35 &7 &0  \\ \hline
\end{tabular}
\end{center}
\end{comment}
% \vskip .2in

 Let $Lie_n$ denote the $S_n$-module obtained by inducing a primitive $n$th root of unity from the cyclic subgroup generated by an $n$-cycle up to $S_n.$  It is a well-known fact that $Lie_n$ is also the representation of the symmetric group acting on the multilinear component of the free Lie algebra (\cite[Ex. 7.88-89]{St4EC2}). Write $Lie$ for $\sum_{n\geq 1} {\rm ch\,}Lie_n.$ 
Recall  from the Introduction that we denote by  $f_n$  the conjugacy action on the $n$-cycles, and that $f[g]$ denotes plethysm.  The functions $Hk_n$ satisfy an interesting plethystic identity.  In order to establish this, we need the following connection between $Lie$ and the conjugacy action.  In keeping with the notation of \cite{Su3}, we will write $Conj_n$ for $f_n$ in the remainder of this section.

\begin{proposition}\label{prop2.19}  \cite[Theorem~2.23]{Su3} %corrected from Prop 6.6
 \begin{equation*}\sum_{m\geq 1} p_m[Lie]=\sum_{n\geq 1} Conj_n.
\end{equation*}
\end{proposition}
 % 2018-12-17: Referee called this next proposition a gem!
 \begin{proposition}\label{prop2.20} Let $W_n$ be the representation with characteristic  $Hk_n.$  Then $W_n$ satisfies the following properties:
  \begin{enumerate}[label=(\arabic*)]
  \item The restriction $W_{n+1}\downarrow_{S_n}$ from $S_{n+1}$ to $S_n$ is isomorphic to 
  the direct sum of $W_n$ and the induced module $(W_n\downarrow_{S_{n-1}})\uparrow^{S_n}.$
  \item %Let $Lie_n$ denote the $S_n$-module obtained by inducing a primitive $n$th root of unity from the cyclic subgroup generated by an $n$-cycle up to $S_n.$  Write $Lie$ for $\sum_{n\geq 1} {\rm ch\,}Lie_n.$  Then 
  $Conj_n$ is the degree $n$ term in 
  $$ (1-Lie)\cdot \sum_{n\geq 1}Hk_n[Lie].$$
 \end{enumerate}
 \end{proposition}
 \begin{proof} 
 The following symmetric function identity is immediate from the definition of $Hk_n:$
 \begin{equation}\label{eq2.9} \sum_{n\geq 1}p_n = (1-p_1) \sum_{n\geq 1} Hk_n. \end{equation}
 Taking partial derivatives with respect to $p_1$ gives 
 $$\frac{\partial}{\partial p_1}Hk_{n+1} = Hk_n +p_1\frac{\partial}{\partial p_1}Hk_{n},$$ 
 which is (1).
 
 Taking the plethysm of both sides of eqn.~\eqref{eq2.9} with $Lie,$ and invoking Proposition~\ref{prop2.19}, now gives (2).   \end{proof}

\section{The representations $\psi_{2^k}$ and the twisted conjugacy action}

The functions $\psi_{2^k}$ of Theorem~\ref{thm2.11} appear to have interesting properties.  We state separately the following consequence of  the proof of Theorem~\ref{thm2.11}: 
\begin{corollary}
$\psi_{2m}-h_2 p_1^{2m-2}$ is Schur-positive, and hence so is 
$\psi_{2m}-h_2^r p_1^{2m-2r},$ $1\leq r \leq m$.
\end{corollary}
\begin{proof}  The second statement follows by induction from the  Schur positivity of expression~\eqref{eq2.1}, upon writing 

$\psi_{2m}-h_2^{r} p_1^{2m-2r}=\psi_{2m}-h_2^{r-1}(p_1^2-e_2) p_1^{2m-2r}=(\psi_{2m}-h_2^{r-1}p_1^{2m-2(r-1)})+e_2 p_1^{2m-2r}.
$
\end{proof}
In fact the following stronger statement appears to be true.
\begin{conjecture}\label{conj2} For $k\geq 2,$ $\psi_{2^k}-2h_2^2 p_1^{2k-4}$ is Schur-positive.  This has been verified for $k\leq 16.$ 
\end{conjecture}
\begin{remark}\label{rk3.2}  Note however  that $\psi_{2^k}-2h_2 p_1^{2k-2}$ 
is NOT Schur-positive.  This can be easily verified by computation, using Theorem~\ref{thm2.11}, for $k= 2,3,4.$
\end{remark}
However, we do have the following:
\begin{lemma}\label{lem3.3}  The symmetric function $\psi_{2^m}-h_2\psi_{2^{m-1}}$ is Schur-positive.  More generally, for $k\leq m,$ the function 
$\psi_{2^m}-h_2^k \psi_{2^{m-k}}$ is Schur-positive.
\end{lemma}
\begin{proof}  Eqn.~\ref{eq2.2} of Theorem~\ref{thm2.11} gives
\begin{align*}\psi_{2^m}-h_2\psi_{2^{m-1}}
&=\sum_{\stackrel{j=1} {j \text{ odd}}}^{m+1} {{m+1}\choose {j}} h_2^{m+1-j} e_2^{j-1}
-h_2 \cdot \sum_{\stackrel{k=1} {k \text{ odd}}}^{m} {{m}\choose {k}} h_2^{m-k} e_2^{k-1}\\
&= \sum_{\stackrel{j=1} {j \text{ odd}}}^{m+1} {{m+1}\choose {j}} h_2^{m+1-j} e_2^{j-1}
-  \sum_{\stackrel{k=1} {k \text{ odd}}}^{m} {{m}\choose {k}} h_2^{m-k+1} e_2^{k-1}\\
&=\sum_{\stackrel{j=1} {j \text{ odd}}}^{m} \left[{{m+1}\choose {j}}-{m\choose j}\right] h_2^{m+1-j} e_2^{j-1} +e_2^m \text{ Odd}(m+1)\\
&=\sum_{\stackrel{j=1} {j \text{ odd}}}^{m} {m\choose j-1} h_2^{m+1-j} e_2^{j-1} +e_2^m \text{ Odd}(m+1),
\end{align*}
and this is clearly Schur-positive.  (As in the proof of Theorem~\ref{thm2.11},  $\text{ Odd}(m+1)$ is 1 if $m+1$ is odd and zero otherwise.)

The more general statement follows from the telescoping sum 
$$\psi_{2^m}-h_2^k\psi_{2^{m-k}}=
\sum_{i=1}^k h_2^{i-1}(\psi_{2^{m+1-i}}-h_2\cdot \psi_{2^{m-i}}).$$
\end{proof}

%%%%%%%%%%%%*****
\begin{proposition}\label{prop3.4} Let $k,r\geq 1,$ and let $m=\lceil{\frac{r}{2}}\rceil.$ 
(So $m=\frac{r+1}{2}$ if $r$ is odd, and $m=\frac{r}{2}$ if $r$ is even.) We have 
$$\psi_{(3, 2^{k},1^r)} 
=\begin{cases} [\psi_{2^{m+k+1}} -h_2p_1^{r+1}\psi_{2^k}]+  p_1^{r} (2h_3+e_3)\psi_{2^k}, & r \text{ odd};\\
p_1 [\psi_{(2^{m+k+1})} -h_2p_1^{r}\psi_{2^k}] +
p_1^{r}(2h_3+e_3) \psi_{2^k}, & r \text{ even}.
\end{cases}$$

%and hence $\psi_{(3, 2^{k},1^r)}$ is Schur-positive for all $r\geq 0$ and all $k\geq 1.$
\end{proposition}

\begin{proof}  We have the recurrence 
$$\psi_{(3, 2^{k},1^r)}-\psi_{(3, 2^{k-1},1^{r+2})}
=p_3 p_1^r p_2^k.$$
Iterating this,  the last two lines of this recurrence are 
$$\psi_{(3, 2,1^{r+2k-2})}-\psi_{(3,1^{r+2k})}=p_3 p_1^{r+2k-2} p_{2},$$ and
$$\psi_{(3,1^{r+2k})}-\psi_{(2^{m}, 1^a)}=p_3 p_1^{r+2k} ,$$
where in the last line $a=0$ if $r+2k+3=2m$ is even, i.e. if $r$ is odd, and $a=1$ if 
$r+2k+3=2m+1$ is odd, i.e. if $r$ is even.

 This telescoping sum collapses to give
 $$\psi_{(3, 2^{k},1^r)} -\psi_{(2^{m}, 1^a)}
 =p_3 p_1^r \cdot (p_2^k +{p_2}^{k-1} p_1^2 +\ldots +p_1^{2k}) =p_3 p_1^r \cdot \psi_{2^k}.$$
 
 But $p_3=h_3-s_{(2,1)}+e_3=2h_3+e_3-h_2h_1.$ 
 Hence we have  
 $$\psi_{(3, 2^{k},1^r)} -\psi_{(2^{m}, 1^a)}=(2h_3+e_3)p_1^r\psi_{2^k}-h_2p_1^{r+1}\psi_{2^k}.$$
 The proposition follows.  \end{proof}
 
 This leads us to make the following conjecture, which has been shown to be true in Lemma~\ref{lem3.3} for $m=1:$ 
 \begin{conjecture}\label{conj3} Let $k,m \geq 1.$ Then $\psi_{2^{k+m}}-h_2 p_1^{2m-2} \psi_{2^k}$ is Schur-positive, and hence so is 
 $\psi_{2^{k+m}}-h_2^r p_1^{2m-2r} \psi_{2^k}, \ 1\leq r\leq m.$ We have verified this for $1\leq k,m\leq 5.$
 \end{conjecture}
 
In view of Proposition~\ref{prop3.4},  the truth of this conjecture would immediately imply Schur positivity of $\psi_{(3, 2^k, 1^r)}$ for all $r,k\geq 1.$ 

\begin{remark}\label{rk3.5} In contrast to Conjecture~\ref{conj2}, computations show that $\psi_{2^{k+m}}-2h_2 p_1^{2m-2} \psi_{2^k}$ is NOT Schur-positive.
\end{remark}

\begin{lemma}\label{lem3.6} The function  $g_4=p_3p_1+h_2^2$ is Schur-positive.
\end{lemma}
\begin{proof} It is easily verified, using the expansion $p_3=h_3-s_{(2,1)}+e_3,=h_3+e_3-(h_2h_1-h_3),$  that 
$$p_3p_1+h_2^2=p_1(2h_3+e_3)-h_2 p_1^2 +h_2^2=
 p_1(2h_3+e_3)-h_2 e_2=2 s_{(4)}+s_{(3,1)}+s_{(1^4)}.$$
\end{proof}

We are able to settle the following special cases:

\begin{proposition}\label{prop3.7}  Let $r=0,1,2.$ Then $\psi_{(3,2^{k},1^r)}$ is Schur-positive.
\end{proposition}
\begin{proof} We use Proposition~\ref{prop3.4}. 
First let $r=1.$ Then we have 
$$\psi_{(3,2^{k},1)}= [\psi_{2^{k+2}}-h_2 p_1^2\psi_{2^k}] + p_1(2h_3+e_3)\psi_{2^k}
=[\psi_{2^{k+2}}-h_2^2 \psi_{2^k}] +g_4\psi_{2^k},$$
where $g_4=p_1(2h_3+e_3)-h_2e_2.$ and is thus Schur-positive by Lemma~\ref{lem3.6}.
But the last expression in brackets is also Schur-positive by Lemma~\ref{lem3.3}.

If $r=0$,  Propostion~\ref{prop3.4} reduces to 
$$\psi_{(3,2^k)}=p_1[\psi_{2^{k+1}}-h_2 \psi_{2^k}] +\psi_{2^k} (2 h_3+e_3),$$
and again this is Schur-positive by Lemma~\ref{lem3.3}.

Finally if $r=2,$ Proposition~\ref{prop3.4} gives 
$$\psi_{(3,2^k, 1^2)}=p_1[\psi_{2^{k+2}}-h_2 p_1^2 \psi_{2^k}] +p_1^2 (2h_3+e_3)\psi_{2^k}
=p_1[\psi_{2^{k+2}}-h_2^2  \psi_{2^k}]
+p_1g_4\psi_{2^k};$$
invoking Lemmas~\ref{lem3.3} and~\ref{lem3.6}, this is Schur-positive as before. \end{proof}

This argument fails for $r=3.$ Proposition~\ref{prop3.4} then gives 
$$\psi_{(3,2^k, 1^3)}=[\psi_{2^{k+3}}-h_2 p_1^4 \psi_{2^k}] +p_1^3 (2h_3+e_3)\psi_{2^k}
=[\psi_{2^{k+3}}-h_2^3  \psi_{2^k}]
+g_6\psi_{2^k},$$
but the function $g_6=p_1^3(2h_3+e_3)-h_2e_2(h_2+p_1^2)$ is no longer Schur-positive.

In previous work of this author, a sign-twisted conjugacy action of $S_n$ was defined in terms of the exterior powers of the conjugacy action, and the following analogue of Theorem~\ref{thm1.1}  was established (recall that $f_n$ is the characteristic of the conjugacy action on the class of $n$-cycles):
\begin{theorem}\label{thm3.8}\cite[Theorem 4.2]{Su1} The twisted conjugacy action has Frobenius characteristic $\varepsilon_n$ satisfying 
\begin{enumerate}[label=(\arabic*)]
\item $\varepsilon_n=\sum_{\lambda\vdash n}\prod_i e_{m_i}[f_i],$ where $\lambda$ has $m_i$ parts equal to $i;$ 
\item $\varepsilon_n
=\sum_{\stackrel {\lambda\vdash n} {\text{all parts odd}}} p_\lambda;$  hence the latter sum is  Schur-positive.
\end{enumerate}
\end{theorem}

Note that $\varepsilon_n$ is self-conjugate, so in particular the multiplicities of the trivial and sign representations coincide (and are equal to the number of partitions of $n$ with all parts odd).  
Based on character tables up to $n=10,$ we were led to make a conjecture in the spirit of Conjecture~\ref{conj1}, which we have subsequently verified for $n\leq 28$.

Let $\mu\vdash n$ be a partition with all parts odd.  Define 
$$\varepsilon_\mu=\sum_{\stackrel {(1^n)\leq \lambda\leq \mu} {\text{all parts odd}}} p_\lambda.$$  
\begin{conjecture}\label{conj4} Let $\mu\vdash n$ be a partition with all parts odd. The symmetric function $\varepsilon_\mu$ is Schur-positive.
\end{conjecture}
Note that $\varepsilon_\mu$ is necessarily self-conjugate.

Theorem~\ref{thm3.8} says that $\varepsilon_{(n)}$ for $n$ odd and $\varepsilon_{(n-1,1)} $ for $n$ even are Schur-positive, since  we now have 
$$\varepsilon_n=\begin{cases} \varepsilon_{(n)}, &n \text{ odd,}\\
                                            \varepsilon_{(n-1,1)}, &n \text{ even.}\end{cases}$$

The chains in reverse lexicographic order are now as follows:
\begin{enumerate}[label=(\arabic*)]
\item If $n$ is odd:
$$(n)>(n-2,1^2)>(n-4,3,1)>(n-4,1^4)>(n-6,3^2)>(n-6,3,1^3)>(n-6, 1^6)>\ldots$$
\item If $n$ is even:
$$(n-1,1)>(n-3,3)>(n-3,1^3)>(n-5,3,1^2)>(n-5,1^5)>\ldots$$
\item
At the bottom of the chain we always have
$$(1^n)<(3, 1^{n-3})<(3^3, 1^{n-6})<\ldots <(3^{\lfloor\frac{n}{3}\rfloor}, 1^r)<(5,1^{n-5})<\ldots$$
where $r\equiv\! n \mod 3,\, 0\leq r\leq 2.$
\end{enumerate}

Some cases of Conjecture~\ref{conj4} are easy to establish, e.g. for $\mu=(3, 1^{n-3}),$ $p_1^n +p_3p_1^{n-3}$ is clearly Schur-positive.  More generally we have 
\begin{proposition}\label{prop3.9} If $\mu=(3^r, 1^{n-3r})$ then $\varepsilon_\mu=\sum_{\stackrel {(1^n)\leq \lambda\leq \mu} {\text{all parts odd}}} p_\lambda$ is Schur-positive.
\end{proposition}

\begin{proof} Note that  $\varepsilon_{(3^r, 1^{n-3r})} =p_1^{n-3r} 
\varepsilon_\nu$ where $\nu=(3^r).$   But $\varepsilon_\nu$ is the sum of power sums for $\lambda$ in the set $T_{3r}$ consisting of all partitions with parts equal to 1 or 3.  By \cite[Theorem 4.23]{Su1}, this is Schur-positive.
\end{proof}

An analogue of Theorem~\ref{thm2.1} holds here as well.  It is also a consequence of Theorem~\ref{thm2.2}, since the global classes defined there are also conjugacy classes appearing in $\varepsilon_n.$
\begin{theorem}\label{thm3.10}\cite[Theorem 4.9, Proposition 4.22]{Su1} The representation $\varepsilon_n$ contains all irreducibles.  The multiplicity of the trivial representation (and hence also the sign) is the number of partitions of $n$ into odd parts.  In particular this multiplicity is at least $\lceil{\frac{n}{2}}\rceil\geq 3$ for $n\geq 5.$
\end{theorem}

The last statement in the theorem is simply a consequence of the observation that if $n$ is odd, the partitions $(n-2r, 1^{2r}), 0\leq r\leq \frac{n-1}{2}$ all have odd parts, while if $n$ is even, the partitions $(n-1-2r, 1^{2r}), 0\leq r\leq \frac{n-2}{2}$ all have odd parts.

\begin{proposition}\label{prop3.11}Let $n$ be odd. Then $\varepsilon_{(n-2,1^2)}, \varepsilon_{(n-4,3,1)}$ 
and $\varepsilon_{(n-4,1^3)}$ are all Schur-positive.  
\end{proposition}
\begin{proof}  
We use Theorem~\ref{thm3.10}. Since  $\varepsilon_n$ contains all irreducibles, from Lemma~\ref{lem2.13} (1), 
$\varepsilon_{(n-2,1^2)}=\varepsilon_n-p_n$ must be Schur-positive.
Similarly, $\varepsilon_{(n-4,3,1)} =\varepsilon_n-p_n-p_{n-2} p_1^2.$
From Lemma~\ref{lem2.13} (1) and (9), $p_n+p_{n-2} p_1^2$ is multiplicity-free except for the occurrence of $2 (n)+2(-1)^n (1^n)=2 (n)-2(1^n).$ But the trivial representation occurs in $\varepsilon_n$ with multiplicity equal to the number of partitions of $n$ with all parts odd, and this is at least 2 for any odd $n.$ Since these representations are all self-conjugate, the proof is complete.

Finally, consider $\varepsilon_{(n-4,1^3)}=\varepsilon_n-(p_n+p_{n-2} p_1^2+p_{n-4}p_3p_1).$  Note that here we have $n\geq 7.$ Observe that from (1), (9) and (10) of Lemma~\ref{lem2.13}, $p_n+p_{n-2} p_1^2+p_{n-4}p_3p_1$ is multiplicity-free except for the occurrence of $3 (n)+3(-1)^n (1^n)=3 (n)-3(1^n).$  The result now follows as before from Theorem~\ref{thm3.10}.
\end{proof}.
\begin{proposition}Let $n$ be even. Then $\varepsilon_{(n-3,3)}$ and 
 $\varepsilon_{(n-3,1^3)}$ are Schur-positive.
\end{proposition}
\begin{proof} We have $\varepsilon_{(n-3,3)}=\varepsilon_n-p_1p_{n-1}$, so the result follows again from Theorem~\ref{thm3.10} and Lemma~\ref{lem2.13} (2).

Next we have $\varepsilon_{(n-3,1^3)}=\varepsilon_n-p_1p_{n-1}-p_3p_{n-3}.$ From Lemma~\ref{lem2.13} we know that $p_1p_{n-1}-p_3p_{n-3}$ is multiplicity-free except for the term $2((n)+(1^n)).$ But for $n$ even, $n\geq 6,$ there are at least three partitions with odd parts, namely $(n-1,1),$ $(3,1^{n-3})$ and $(n-3,3)$. This ensures the trivial representation (and hence the sign, since the representations are self-conjugate) occurs with multiplicity at least 3 in $\varepsilon_n,$ and hence with positive multiplicity in $\varepsilon_{(n-3,1^3)},$ completing the argument.
\end{proof}
Recall that $\omega$ denotes the involution on the ring of symmetric functions which sends $h_n$ to $e_n.$ Another result of \cite{Su1} states that
\begin{theorem}\label{thm3.13}\cite[Theorem 4.11]{Su1} The sum $\sum_{\stackrel {\lambda\vdash n} {n-\ell(\lambda) {\text{ even}}}} p_\lambda$ equals $\frac{1}{2}(\psi_n+\omega(\psi_n))$ and  is  Schur-positive.
\end{theorem}
Similarly we have, for any partition $\mu$ of $n,$ 
$\frac{1}{2}(\psi_\mu+\omega(\psi_\mu))=\sum_{\stackrel {(1^n)\leq \lambda\leq \mu} {n-\ell(\lambda) {\text{ even}}}} p_\lambda.$ This leads us to make the following conjecture, which has been  verified for $n\leq 20:$ 
\begin{conjecture}\label{conj5}  Let $\mu\vdash n.$ The sum $\sum_{\stackrel {(1^n)\leq \lambda\leq \mu} {n-\ell(\lambda) {\text{ even}}}} p_\lambda$ is Schur-positive.
\end{conjecture}
Clearly Conjecture~\ref{conj1} implies Conjecture~\ref{conj5}.    
Maple computations with the character table of $S_n$ show that the sum $\sum_{\lambda\in T} p_\lambda$ is NOT Schur-positive for arbitrary subsets $T$ containing $(1^n)$ and consisting of all partitions $\lambda$ with $n-\ell(\lambda)$ even.  The first counterexample occurs only for $n=14,$ and  there are then at least $2^{11}$ such subsets for which Schur positivity fails.  

Note that if we require that $n-\ell(\lambda)$ be odd, but also  include the regular representation in the sum, the preceding conjecture is false:

\qquad \qquad\qquad
 the sum 
$p_1^n+\sum_{\stackrel {(1^n)\leq \lambda\leq \mu} {n-\ell(\lambda) {\text{ odd}}}} p_\lambda$ is not Schur-positive.
 
\begin{comment}

 for k to 5 do ConsecOddPar(k) end do;
                              $s[1]$
                       Schur-positive, $p1$
                       
                             $2 s[2]$
                                           
                    Schur-positive, $p1^2  + p2$
                       $2 s[3] + 2 s[2, 1]$
                                            
                  Schur-positive, $p1^3  + p1 p2$
                                    
                  Schur-positive, $p1^4  + p1^2  p2$
              NEGATIVE COEFF, $[-1, s[1, 1, 1, 1]]$
                            
                       $p1^4  + p1^2  p2 + p4$
                                  
                  Schur-positive, $p1^5  + p1^3  p2$
             NEGATIVE COEFF,$ [-1, s[1, 1, 1, 1, 1]]$
             NEGATIVE COEFF, $[-2, s[1, 1, 1, 1, 1]]$
                                       
                  $p1^5  + p1^3  p2 + p1 p4 + p2 p3$
 
\end{comment}

\begin{question}\label{qn3.14}In  \cite{Su1} and \cite{Su3},  $S_n$-modules are constructed whose characteristics are multiplicity-free sums of power sums, thereby settling the Schur positivity question in these cases.
Is there a representation-theoretic context for  the sums $\psi_\mu$? 
\end{question}
 
 \section{Arbitrary subsets of conjugacy classes}
 
 In this section we examine the following more general question:
 %From RPS October 2015:
Let $f(n)$ be the number of subsets of
$\{p_\lambda : \lambda\vdash n\}$ containing $p_1^n,$ and having the property that the sum of their elements is NOT Schur-positive.  
What can be said about $f(n)?$   Richard Stanley computed the  values of $f(n)$ for $n\leq7$ after seeing a preprint of \cite{Su1}. 
Table 2 extends these values up to $n= 10.$

Recall from Section 1 that $\psi_T$ denotes the Schur function $\sum_{\mu\in T} p_\mu.$ The analysis of the multiplicity of the sign representation in Example~\ref{ex1.4} suggests a way to obtain a lower bound for the numbers $f(n).$ 
Indeed,  let $A(n)=\{\mu\vdash n: n-\ell(\mu) \text{ is even}\},$ and let $B(n)=\{\mu\vdash n: n-\ell(\mu) \text{ is odd}\}.$  Let $\alpha(n), \beta(n)$ respectively be the cardinalities of $A(n), B(n).$  Clearly $\alpha(n)+\beta(n)=p(n).$   As in \cite[Proposition 4.21]{Su1} (see also eqn.~\eqref{1.1}),
\begin{equation}\label{eq4.1}\alpha(n)-\beta(n)=\sum_{\mu\vdash n}(-1)^{n-\ell(\mu)}\end{equation} 
is the number of self-conjugate partitions of $n,$ and hence 
$\alpha(n)\geq \beta(n).$

By manipulating generating functions it can be seen that $\alpha(n)$ is also the number of partitions of $n$ with an even number of even parts, and an arbitrary number of odd parts.  The sequence appears in \cite[A046682]{O}.

\begin{comment}
Let $p(n,k)$ denote the number of partitions of $n$ into $k$ parts.
The generating function is 
$\sum_{n,k}p(n,k)t^k q^n=\prod_i(1-tq^i)^{-1},$
and hence

\begin{center} 
$\sum_{n,k}p(n,k)(-1)^{k+n}t^k q^n=\prod_i(1+t (-1)^iq^i)^{-1}.$\end{center}

But
\begin{center}$\alpha(n)=\sum_{k\equiv n\!\!\mod 2} p(n,k),$\end{center} and hence the generating function for $\alpha(n)$ is given by
\begin{align*}2 \sum_{n\geq 0} \alpha(n) q^n&=\prod_i(1-q^i)^{-1}
+\prod_i(1+(-1)^i q^i)^{-1}\\
&=\prod_j (1-q^{2j-1})^{-1}[\prod_j(1-q^{2j})^{-1}+
\prod_j(1+q^{2j})^{-1}]\\
&=\prod_j (1-q^{2j-1})^{-1}[\sum_{\stackrel {\mu}{\text{all parts even}}} q^{|\mu|} +\sum_{\stackrel {\mu}{\text{all parts even}}} q^{|\mu|} (-1)^{\sum_{m_i \text{ odd}} m_{i}(\mu)}]\\
&=\prod_j (1-q^{2j-1})^{-1}[\sum_{\stackrel {\mu}{\text{all parts even}}} q^{|\mu|} +\sum_{\stackrel {\mu}{\text{all parts even}}} q^{|\mu|} (-1)^{\sum_{m_i } m_i(\mu)}]\\
&=\prod_j (1-q^{2j-1})^{-1}[\sum_{\stackrel {\mu}{\text{all parts even}}} q^{|\mu|} +\sum_{\stackrel {\mu}{\text{all parts even}}} q^{|\mu|} (-1)^{\ell(\mu)}]\\
&=\prod_j (1-q^{2j-1})^{-1}\sum_{\stackrel {\mu}{\text{all parts even}}} q^{|\mu|} (1+(-1)^{\ell(\mu)})\\
&=2\cdot \prod_j (1-q^{2j-1})^{-1}\sum_{\stackrel {\mu}{\text{all parts even}}} q^{|\mu|}.
\end{align*}

In other words, $\alpha(n)$ is also the number of partitions of $n$ with an even number of even parts, and an arbitrary number of odd parts.  The sequence appears in \cite[A046682]{O}.

\end{comment}

\begin{proposition}\label{prop4.1}  Let $T$ be a subset of the set of partitions of $n$ \textit{not} containing the partition $(1^n)$.  The Schur function indexed by $(1^n)$ appears with negative multiplicity in the Schur expansion of $\psi_{T\cup\{(1^n)\}}$  if and only if $|T\cap B(n)|\geq 2+|T\cap A(n)|.$
Hence the number of such subsets gives  
 the following lower bound for $f(n)$:

$$\ell b(n)=\sum_{i=0}^{p(n)-\alpha(n)-2} \binom{p(n)-1}{i}.  $$
In particular $f(n)$ is positive for all $n\geq 4.$
\end{proposition}

\begin{proof}  This is immediate from the preceding discussion and the fact that each $p_\mu$ contributes $(-1)^{n-\ell(\mu)}$  to the multiplicity of $(1^n)$ in
$\psi_T$.  A simple count then tells us that this multiplicity is negative for exactly as many subsets $T$ as given by the following sum:

$$\sum_{a=0}^{\alpha(n)-1}\binom{\alpha(n)-1}{a}
\sum_{b=a+2}^{\beta(n)} \binom{\beta(n)}{b}.$$

This is precisely the sum of the coefficients of the powers of $x^j, j\geq 2$,  in the Laurent series expansion of 
$$(1+x^{-1})^{\alpha(n)-1} (1+x)^{\beta(n)}
=x^{-(\alpha(n)-1)} (1+x)^{\alpha(n)+\beta(n)-1}.$$
Since $\alpha(n)+\beta(n)=p(n),$ this in turn is the sum of the coefficients of the terms $x^j$ in $(1+x)^{p(n)-1},$ for $j\geq \alpha(n)+1,$   i.e:
$$\sum_{j=\alpha(n)+1}^{p(n)-1} \binom{p(n)-1}{j}.$$

Now replace $j$ with $p(n)-1-i.$ The last claim follows because $\alpha(n)$ is the number of partitions with an even number of even parts and thus $\alpha(n)\leq p(n)-2.$ (If $n\geq 4,$ exclude the partitions  
$(2, 1^{n-2})$  and $(n)$ if $n$ is even, $(n-1,1)$ if $n$ is odd.)
\end{proof}

%1, 1, 1, 2, 3, 4, 6, 8, 12, 16, 22, 29, 40, 52, 69, 90, 118, 151
%Oeis A046682

%The first few values of $\alpha(n), 1\leq n\leq 17,$ are 
%$1,1,2,3,4,6,8,12,16,22,29,40,52,69, 90, 118, 151.$

%The first few values of $\ell b(n), n\geq 1, $ are:
%$0,0,0,1,\ell b(5)=7, \ell b(6)=176, \ell b(7)=3473, \ell b(8)=401,930; \ell b(9)=123,012,781; \ell b(10)=585,720,020,356.$

Table 2 includes data up to $n=10,$ and the resulting   lower bound $\ell b(n)$ on the number $f(n)$ of non-Schur-positive functions $\psi_T,$
 omitting the trivial values   $f(n)=0$ for $n\leq 3.$
%Since the number of subsets of $Par(n)$ which contain $(1^n)$ is $2^{p(n)-1},$ where $p(n)$ is the partition function, we have:

%\vskip.1in

%\begin{center}{\small Table 2}\end{center}
\begin{center}{\small Table 2}

%\begin{tabular}{|c|c|c|c|c|c|c|c|c|c|c|}\hline
\begin{tabular}{|c|c|c|c|c|c|c|c|}\hline
$n$  &4 &5 & 6 & 7 &8 &9 &10\\ \hline
$p(n)$  & 5 & 7 &11 & 15 &22 &30 &42\\ \hline
${\bf f(n)}$   &{\bf 1} &{\bf  7} &{\bf  184} &{\bf  3674} &{\bf 488,259} &{\bf 145,796,658} &${\bf\scriptstyle 670, 141, 990, 673}$\\ \hline
$\ell b(n)$   &1 &7 & 176 & 3473 &401,930 & 123,012,781 &${\scriptstyle 585, 720, 020, 356}$\\ \hline 
$\frac{f(n)}{2^{p(n)-1}}$  &$0.06$
&$0.11$ &$0.18$ &$0.22$ &$0.23$  & $0.272 $ & $ 0.305   $\\ \hline
%.2715674378  %.3047452950
$\frac{\ell b(n)}{2^{p(n)-1}}$   &$0.06$
&$0.11$ &$ 0.172 $ &$  0.212 $ &$ 0.192 $  & $ 0.229 $ 
& $0.266$\\ \hline
%$\frac{f(n)}{2^{p(n)-1}}$ & 0 & 0 & 0 &$\frac{1}{16}=.06$ 
%&$\frac{7}{64}=.11$ &$\frac{23}{128}=.18$ &$\frac{1837}{8192}=.22$ 
%&$\frac{488259}{2^{21}}=.23$\\ \hline
%$\frac{\ell b (n)}{f(n)}$ & & & &1 & 1 &0.956  &0.945 & 0.823 
%&$0.844 $  &\\ \hline
%.8437284001
\end{tabular}
\end{center}
%\end{remark} 

\vskip.1in
\begin{proposition}\label{prop4.2} There exists a subset $T$ of the set of partitions of $n$, with $(1^n)\in T$, such that
\begin{itemize}
\item the irreducible  $(n-1,1)$ appears with negative multiplicity in $\psi_T$,    if and only if $n\geq 10.$   
\item the irreducible  $(2,1^{n-2})$ appears with negative multiplicity in $\psi_T,$   if and only if $n\geq 6.$
\end{itemize}
\end{proposition}

\begin{proof}  Write $\chi^{\mu}$ for the irreducible character indexed by the partition $\mu.$ Recall that  the value of $\chi^{(n-1,1)}(\lambda) $ is one less than the number $m_1(\lambda)$ of parts of $\lambda$ which are equal to 1, and is therefore never less than $-1. $ Hence we have 

$\chi^{(n-1,1)}(\lambda) =\begin{cases}
-1 &\text{ for the } (p(n)-p(n-1)) \text{ partitions } \lambda \text{ with } m_1(\lambda)=0,\\
0 &\text{ for the} (p(n-1)-p(n-2)) \text{ partitions } \lambda \text{ with } m_1(\lambda)=1,\\
\geq 1 &\text{ for the } p(n-2) \text{ partitions } \lambda \text{ with } m_1(\lambda)\geq 2.
\end{cases}$
\vskip.2in
\begin{comment}
$|\{\mu\vdash n:m_1(\mu)\geq 1\}|=p(n-1) (\text{ easy bijection: remove 1 from }\mu$) 

and hence $|\{\mu\vdash n:m_1(\mu)< 1\}|=p(n)-p(n-1).$

$|\{\mu\vdash n: m_1(\mu)\geq 2\}|=|\{\mu\vdash n-1:m_1(\mu)\geq 1\}|=p(n-2)$

$|\{\mu\vdash n:m_1(\mu)= 1\}|=|\{\mu\vdash n:m_1(\mu)\geq 1\}|-|\{\mu\vdash n: m_1(\mu)\geq 2\}|=p(n-1)-p(n-2).$

\end{comment}

Consider the conjugacy classes indexed by the $p(n)-p(n-1)$ partitions with no part equal to 1, and the partition $(1^n)$.
The row sum indexed by $(n-1,1)$ in the character table of $S_n$ will then be $n-1-(p(n)-p(n-1)).$  
The first claim follows  by observing that $p(n)-p(n-1)$ first exceeds $\chi^{(n-1,1)}(1^n)=n-1$ when $n=10,$ and the fact that the values $p(n)-p(n-1)$ increase.

Of course we could also append to the set $T$ above any of the $2^{p(n-1)-p(n-2)}$ subsets of conjugacy classes with exactly one fixed point (since these do not contribute to the multiplicity of 
$(n-1,1)$), to obtain even more non-Schur-positive instances of $\psi_T;$  for  the number of subsets with negative multiplicity 
for $(n-1,1)$ this gives a lower bound of
\begin{equation}2^{p(n-1)-p(n-2)} \sum_{j=0}^{p(n)-p(n-1)-n} \binom{p(n)-p(n-1)}{n+j}.\end{equation}

Next we note that $\chi^{(2,1^{n-2})}(\lambda)=(-1)^{n-\ell(\lambda)} \chi^{(n-1,1)}(\lambda),$ since the two irreducibles are conjugate.
Thus the number of times that $\chi^{(2,1^{n-2})}(\mu)$ equals $(-1)$ is 
\begin{equation*}
\begin{split}
&|\{\mu\vdash n: n-\ell(\mu) \text{ is even and }\mu \text{ has no singleton parts}\}|\\&+|\{\mu\vdash n: n-\ell(\mu) \text{ is odd and }\mu \text{ has exactly two singleton parts}\}|.
\end{split}
 \end{equation*}
  Similarly 
 the number of times that $\chi^{(2,1^{n-2})}(\mu)$ equals $(-r), r\geq 2,$ is 
 %
%\begin{equation*}
%\begin{split}
%&
$$|\{\mu\vdash n: n-\ell(\mu) \text{ is odd and }\mu \text{ has at least three singleton parts}\}|. $$ %&
 %=|\{\mu\vdash n-3: n-\ell(\mu) \text{ is odd}\}|,
% \end{split}
% \end{equation*}
 %and this is simply $p(n-3)-\alpha(n-3).$
 
 Combining these two quantities, we have that the number of conjugacy classes for which the value of $\chi^{(2,1^{n-2})}$ is negative is $|C_1|+|C_2|,$ where 
 $$C_1=\{\mu\vdash n: n-\ell(\mu) \text{ is even and }\mu \text{ has no singleton parts}\}$$
 and $$ C_2=\{\mu\vdash n: n-\ell(\mu) \text{ is odd, }\mu \text{ has at least \textit{two} singleton parts}\}.$$
 
 The set  $C_2$ is in bijection with the set of all odd-signature partitions of $n-2$, so has cardinality $p(n-2)-\alpha(n-2).$ Also, the character values on the classes $(1^n)$ and $(2,1^{n-2})$ together add up to $(n-1)-(n-3)=2.$ Hence, by choosing at least 3 additional conjugacy classes in $C_2$, excluding the partition $(2,1^{n-2})$,  we obtain that the multiplicity of $\chi^{(2,1^{n-2})}$ is negative for at least 
 \begin{equation}2^{p(n-1)-p(n-2)} \sum_{j\geq 3} 
 \binom{p(n-2)-\alpha(n-2)-1}{j}\end{equation}
 subsets, and this is positive as soon as $n\geq 6,$   
 since then $p(n-2)-\alpha(n-2)\geq p(n-2)-p(n-1)\geq 2.$
 Likewise the cardinality of $C_1$ is $p(n-2)-\alpha(n-2).$ 
\end{proof}

Note that the lower bound of Proposition~\ref{prop4.1}  surpasses the two lower bounds obtained above.  In order to test the Schur positivity of $\psi_T$, we need to examine the multiplicity of the irreducible indexed by each $\lambda\vdash n$ in $\psi_T.$ This is given by 
$a_T(\lambda)=f^\lambda + \sum_{\mu \in T:\mu\neq (1^n)} \chi^\lambda(\mu),$ where $f^\lambda=\chi^\lambda((1^n))$ is the number of standard Young tableaux of shape $\lambda.$   Now $
\chi^\lambda((1^n))$ is larger than any other value of the character 
$\chi^\lambda.$
% For any finite group $G$ and any complex representation $\phi:G\mapsto Gl(V),$ with $dim(V)=d,$ the eigenvalues of the $d$ by $d$  matrix $\phi(g)$ are $|G|$th roots of unity.  By the triangle inequality, $|tr (\phi(g))|=\text{ sum of eigenvalues of }\phi(g)$ is less than or equal to the sum of the absolute values of its eigenvalues, and is thus less than or equal to $d=|tr (\phi(e))|.$
Hence one way in which we can see how to make these values negative is to find $\mu$ such that  $f^\lambda+\chi^\lambda(\mu)$ is small relative to the number $p(n)$ of conjugacy classes.  For instance:

\begin{proposition}\label{prop4.3} Let $\lambda   \vdash n, $  and let $\tau=(2,1^{n-2})$ be the (conjugacy class of) a single transposition.  Then 
$\chi^\lambda ((1^n)) +\chi^\lambda(\tau)$ equals
\begin{enumerate}[label=(\arabic*)]
\item 2 if $\lambda=(2,1^{n-2}),$ 
\item $2(n-3)$ if $\lambda=(2^2,1^{n-4}),$ 
\item $(n-2)(n-5)$ if $\lambda=(2^3,1^{n-6}),$ 
\item $2(n-2)$ if $\lambda=(3,1^{n-3}).$
\end{enumerate}
\end{proposition}

\begin{proof}  The first part has already been observed in the proof of Proposition~\ref{prop4.2}. For the rest, we use the formula 
$\chi^\lambda(\tau)= \frac{f^\lambda}{{n\choose 2}} (b(\lambda \rq)-b(\lambda)),$ where $b(\lambda)=\sum_i (i-1)\lambda_i,$  as well as the hook length formula for $f^\lambda.$ (See 
\cite[Ex. 7.51]{St4EC2}). When $\lambda \rq$ dominates $\lambda$ we must have $b(\lambda \rq)<b(\lambda) $ and thus $\chi^\lambda(\tau)$ is negative.
\end{proof}
 An examination of the character tables of $S_n$ leads to  the following observations.     The  use of  character tables eliminates the need for Stembridge's SF package for Maple, by means of which  the values $f(n)$ were originally calculated, up to $n=8.$
\begin{itemize}
\item For $n=6,$ of the 184 subsets that fail to be Schur-positive, exactly 176 fail to be Schur-positive because of the irreducible $(1^6)$,  another 4 fail because the irreducible $(2,1^4)$ appears with negative coefficient, and the remaining 4 fail because of the irreducible $(3^2).$    From the character table of $S_6,$ it is easy to identify these 8 subsets.  (In each of these cases no other irreducibles occur with negative coefficient.)

\item For $n=7,$ the number of subsets failing Schur positivity because of (a negative coefficient for) $(1^7)$ is $3473,$ and  384 were identified as failing (in part) because of the irreducible $(2, 1^5).$  The count for subsets in which both irreducibles appear with negative coefficient is 183, and this confirms $f(7)=3674.$ 
From the character table of $S_7,$ it is easy to verify that the number of subsets $T$ resulting in a negative coefficient for $(2,1^5)$ in $\psi_T$ is exactly 384, and  also that no other irreducibles occur with negative coefficient in any subset.
%Since Lemma 4.1 gives $\ell b(7)=3473,$ we deduce that 3473-3290=183 of these 384 have a negative coefficient for \textit{both} $(1^7)$ and $(2, 1^5).$ 

\item For $n=8,$ by examining the negative entries in the character table, we see  that for any subset $T$ containing $(1^n)$, the only two possibilities for negative coefficients in the Schur expansion of $\psi_T$ are $(1^8)$ and $(2,1^6).$  There are  $\ell b(8)=401,930$ subsets with negative multiplicity for $(1^8)$,  $153,008$ subsets with negative multiplicity for $(2,1^6),$ 
and  $76,679$ subsets in which \textit{both} irreducibles occur with negative multiplicity.  The reader can check that this agrees with the figure for $f(8)$ in the table.  This computation took 70 seconds in Maple.

\item For $n=9,$ since our lower bound is $\ell b(9)=123,012,781,$ 
we know that $\frac{f(9)}{2^{p(9)-1}}\geq \frac{123,012,781}{2^{29}}=0.22913.$  The character table shows that in addition to $(1^9)$ and $(2, 1^7),$ only the irreducibles $(2^2, 1^5)$ and $(3, 1^6)$ will appear with negative coefficient in some subsets.  The value of $f(9)$ was calculated by exploiting this fact, and took 6.8 hours in Maple. However, the C code ran in only 36 seconds.

\item For $n=10$ similarly, we have 
$\frac{f(10)}{2^{p(10)-1}}\geq \frac{585,720,020,356}{2^{41}}=0.266.$ %$0.26635462761.$
 The character tables now show that one or more of only the following five irreducibles  will appear with negative coefficient in $\psi_T:  (1^{10}), (2,1^8), (2^2, 1^6), (3, 1^7), (9,1),$ 
 for some subset $T$ containing $(1^{10}).$  Again, the computation of $f(10)$ exploited this fact.  It was coded in C, and took 83 hours to produce the result.
 
 \item For $n=11$ the six irreducibles contributing to negativity in $\psi_T$ are 
 
 $(1^{11}), (2,1^9), (2^2,1^7), (2^3,1^5), (3,1^8), (10,1).$ 
 %\item For $n=12,$ there are eight irreducibles:
 %
 %$ (1^{12}), (2,1^{10}), (2^2,1^8), (2^3, 1^6),  (3,1^9), (3^2,1^6), (4,1^8), (11,1).$
\end{itemize}

 Of course this number of irreducibles increases rapidly with $n;$  e.g. for $n=28,$ out of $p(28)=3,718$ partitions,  $89$ can occur with negative multiplicity.  A far more accurate lower bound than $\ell b(n)$ is obtained by taking the number of subsets $T$ in which \textit{either} of the representations $(1^n)$ or $(2,1^{n-2})$ appear with negative multiplicity in $ \psi_T,$ but a formula for this in the spirit of Proposition~\ref{prop4.1} seems difficult to obtain.
 
 Tables 3a and 3b below contain, for each $n,$ the values of the function  $g(n),$ defined to be the number of  partitions $\mu$ of $ n$ such that, for some subset $T$ containing $(1^n),$ the irreducible indexed by $\mu$ appears with negative multiplicity in $\psi_T.$
 
 \begin{center}{\small Table 3a}\end{center}
\begin{center}
%\begin{tabular}{|c|c|c|c|c|c|c|c|c|c|c|}\hline
\begin{tabular}{|c|c|c|c|c|c|c|c|c|c|c|c|c|c|c|c|c|c|c|c|}\hline
$n$  &4 &5 & 6 & 7 &8 &9 &10 &11 & 12  \\ \hline
$p(n)$  & 5 & 7 &11 & 15 &22 &30 &42 &56 &77  \\ \hline
${\bf g(n)}$   &{\bf 1} &{\bf  1} &{\bf  3} &{\bf  2} &{\bf 2} &{\bf 4} &${\bf 5}$ & ${\bf 6}$ &${\bf 8}$ \\ \hline
\end{tabular}
\end{center}
\vskip .1in
 \begin{center}{\small Table 3b}\end{center}
\begin{center}
%\begin{tabular}{|c|c|c|c|c|c|c|c|c|c|c|}\hline
\begin{tabular}{|c|c|c|c|c|c|c|c|c|c|c|c|c|c|c|c|c|c|c|c|}\hline
$n$ &13 &14 & 15 &16 &17 & 18 & 19 & 20 &21 & 22 &23 &24 & 25\\ \hline
$p(n)$ & 101 &135 &176 & 231 & 297 &385 &490 &627& 792 & 1002 & 1255 & 1575 &1958\\ \hline
${\bf g(n)}$   & ${\bf 9}$ & ${\bf 10}$ & ${\bf 10}$  &${\bf 15}$ &${\bf 16}$ &${\bf 22}$  &${\bf 23}$  &${\bf 27}$ &${\bf 33}$&  ${\bf 36}$  &${\bf 43}$ & ${\bf 51}$ &${\bf 56}$\\  \hline
\end{tabular}
\end{center}
%g(26)= 69, g(27)=76, g(28)=89
 \vskip .1in
Based on our computations, we make the following conjecture:
\begin{conjecture}\label{conj6} For $n\geq 6,$ the numbers $\frac{f(n)}{2^{p(n)-1}}$ are bounded below by $\frac{1}{16},$ above by $\frac{1}{2},$ and are strictly increasing.
\end{conjecture}

This would imply an affirmative answer to a question raised by Richard Stanley:
\begin{conjecture}\label{conj7}  The numbers $\frac{f(n)}{2^{p(n)-1}}$ approach a limit strictly between 0 and 1.
\end{conjecture}

%\vskip-.5in
%
%\begin{comment}
We close this section with a list of the cases of Schur positivity known to us, for subsets of conjugacy classes.  

The symmetric function $\psi_T$ is Schur-positive for the following subsets $T$ of the set of partitions of $n.$  In nearly all cases one can describe an $S_n$-module whose Frobenius characteristic is given by $\psi_T.$
\begin{enumerate}[label=(\arabic*)]
\item $\{\mu\vdash n\}$ \cite{So}, \cite[Exercise 7.71]{St4EC2}
\item $\{\mu\vdash n: \mu \text{ has all parts odd}\},$   (\cite[Theorem 4.6]{Su1})
\item $\{\mu\vdash n, \mu \text{ has at least }k\text{ parts equal to 1}\},$ for fixed $k\geq 1.$ (\cite[Corollary 4.10]{Su1})
\item $\{\mu\vdash n: \mu \text{ has all parts odd and at least }k \text{ parts equal to }1 \},$ for fixed $k\geq 1.$   (\cite[Corollary 4.10]{Su1})
\item $\{\mu\vdash n:n-\ell(\mu) \text{ is  even}\},$ (\cite[Theorem 4.11]{Su1})
\item $\{\mu\vdash n:n-\ell(\mu) \text{ is even and } \mu \text{ does not have all its parts odd and distinct}\},$ (\cite[Theorem 4.15]{Su1})
\item $\{\mu\vdash n:\mu  \text{ does not have all parts odd and distinct}\}, n\geq 2,$ (\cite[Theorem 4.11]{Su1})
\item $\{\mu\vdash n:\mu_i =1 \text{ or } k\}$, for fixed $k\geq 2,$  
 (\cite[Theorem 4.23]{Su1})
\item $\{\mu\vdash n:\mu_i \text{ divides } k\}$, for fixed $k\geq 2,$  (\cite[Theorem 5.6]{Su1})
%\item $\{\mu\vdash n:\mu_i \text{ is a power of  }q\}$, for a fixed  prime $q,$ (\cite[Theorem~3.12]{Su3})
%\item $\{\mu\vdash n:\mu_i \text{ is a power of 2, and  }n-\ell(\mu)\text{ is even}\}$,  (\cite[Theorem 2.8]{Su3})
%\item $\{\mu\vdash n:\mu_i \text{ is relatively prime to } q\}$, for a fixed  prime $q$ (\cite[Theorem~3.12]{Su3})
\item The set of partitions $\mu$ of $n$ such that $\mu_i$ has only the primes in $S$ in its prime factorisation, and $\mu$ has  an even number of even parts, for a fixed nonempty subset $S$ of primes (\cite[Corollary 3.6]{Su3}).  
%$\{\mu\vdash n:\mu_i \text{ has only the primes in $S$ in its prime factorisation, with an even number of even parts}\}$,for a fixed nonempty subset $S$ of primes (\cite[Corollary 3.6]{Su3}).  
Three special cases are:
%\item $\{\mu\vdash n:\mu_i \text{ is relatively prime to all  primes in S }\}$, for a fixed nonempty subset $S$ of primes (\cite[Theorem~3.12]{Su3})
\begin{enumerate}
\item $\{\mu\vdash n:\mu_i \text{ is a power of 2, and  }n-\ell(\mu)\text{ is even}\}$,  (\cite[Theorem 2.8]{Su3})
\item  $\{\mu\vdash n:\mu_i \text{ is a power of  }q\}$, for a fixed  prime $q,$ (\cite[Corollary 3.6]{Su3})
\item $\{\mu\vdash n:\mu_i \text{ is relatively prime to } q\}$, for a fixed  prime $q$ (\cite[Corollary 3.6]{Su3})
\end{enumerate}
\item $[(1^n), \mu]$ for any $\mu$ in the interval $[(1^n), (3,1^{n-3})]$ in reverse lexicographic order (Theorem~\ref{thm2.11}, this paper)
\item $[(1^n), \mu]$ for any $\mu$ in the interval $[(n-4,1^4), (n)]$ in reverse lexicographic order (Theorem~\ref{thm2.16}, this paper)
\item $\{(n-k, 1^k): 0\leq k\leq n-1\},$ (Proposition~\ref{prop2.18}, this paper)
\item $[(1^n), \mu]$ for  $\mu=(3, 2^k, 1^r), r=0,1,2,$  (Proposition~\ref{prop3.7}, this paper)
\item $\{(1^n), \mu\}$ for any single $\mu \vdash n.$  (The analogous statement holds trivially for arbitrary finite groups, since for any irreducible $\chi,$ the value of the character $\chi(g)$ is a sum of $\chi(1)$ roots of unity, and hence its absolute value cannot exceed the degree $\chi(1).$ )
\end{enumerate}
%\end{comment}
\section{Additional  tables}

In the tables  below we follow our usual convention of writing simply $\mu$ to signify the Schur function $s_\mu.$ 
\begin{center} \underline{Table 4: Schur function expansion of $\psi_\mu,$ $n\leq 5$} \end{center}
\begin{center}$\psi_1 =(1), \qquad \psi_2=2(2), \qquad \psi_{(1^2)}=(2)+(1^2)$\\
%=p_2+p_1^2
%\psi_{(1^2)}&=p_1^2=s_{(2)}+s_{(1^2)}\\
$\psi_3=3(3)+(2,1)+(1^3),\qquad \psi_{(2,1)} =2 (3)+2(2,1)
, \qquad \psi_{(1^3)}=(3)+2(2,1)+(1^3)$
\end{center}
%=p_2p_1+p_1^3=p_1\psi_2
\vskip-.05in
\begin{align*}
\psi_4&=5(4)+2(3,1)+3 (2^2)+2 (2,1^2)+(1^4)\\
\psi_{(3,1)} &=4(4)+3(3,1)+3(2^2)+(2,1^2)+2(1^4)\\
%=\psi_4-p_4
\psi_{(2^2)}&= 3(4)+3(3,1)+4(2^2)+(2,1^2)+(1^4)\\
%= \psi_{(3,1)-p_3 p_1
\psi_{(2,1^2)}&=2(4)+4(3,1)+2(2^2)+2(2,1^2)\\
%=p_2 p_1^2+p_1^4
\psi_{(1^4)}&=(4)+3(3,1)+2 (2^2)+3 (2,1^2)+(1^4)\\
%\end{align*}
%\begin{align*}
\psi_5&=7(5)+5 (4,1) +6 (3,2) + 5 (3,1^2) + 4 (2^2,1) + 3 (2,1^3) +(1^5) \\
%7*s[5]+5*s[4,1]+6*s[3,2]+5*s[3,1,1]+4*s[2,2,1]+3*s[2,1,1,1]+s[1,1,1,1,1]
\psi_{(4,1)}&= 6(5)+6 (4,1) +6 (3,2) + 4 (3,1^2) + 4 (2^2,1) + 4 (2,1^3) \\
%psi41 := 6 s[5] + 6 s[4, 1] + 6 s[3, 2] + 4 s[3, 1, 1] 
%+ 4 s[2, 2, 1]+ 4 s[2, 1, 1, 1]
\psi_{(3,2)}&= 5(5)+6 (4,1) +7 (3,2) + 4 (3,1^2) + 3 (2^2,1) + 4 (2,1^3) +(1^5)\\
%psi32 := 5*s[5]+6*s[4,1]+7*s[3,2]+4*s[3,1,1]+3*s[2,2,1]+4*s[2,1,1,1]+s[1,1,1,1,1]
%
\psi_{(3,1^2)} &=4(5)+7 (4,1) +6 (3,2) + 4 (3,1^2) + 4 (2^2,1) + 3 (2,1^3) +2(1^5)\\
%psi311 := 4*s[5]+7*s[4,1]+6*s[3,2]+4*s[3,1,1]+4*s[2,2,1]+3*s[2,1,1,1]+2*s[1,1,1,1,1]
%
\psi_{(2^2,1)}&=3(5)+6 (4,1) +7 (3,2) + 4 (3,1^2) + 5 (2^2,1) + 2 (2,1^3) +(1^5)\\
%psi221:=3*s[5]+6*s[4,1]+7*s[3,2]+4*s[3,1,1]
%+5*s[2,2,1]+2*s[2,1,1,1]+s[1,1,1,1,1]
\psi_{(2,1^3)}&= 2(5)+6 (4,1) +6 (3,2) + 6 (3,1^2) + 4 (2^2,1) + 2 (2,1^3)\\
%2*s[5]+6*s[4,1]+6*s[3,2]+6*s[3,1,1]+4*s[2,2,1]+2*s[2,1,1,1]
\psi_{(1^5)}&=(5)+4(4,1)   +5(3,2)   +6(3,1^2) +5(2^2,1) +4(2,1^3) +(1^5)     
\end{align*}
%
%\newpage
\begin{center} \underline{Table 5: Schur function expansion of $\psi_\mu,$ $n=6$} \end{center} \nopagebreak
\begin{align*}
\psi_{(6)}&=
{\scriptstyle 11 (6) + 8 (5, 1) + 15 (4, 2) + 10 (4, 1^2) + 4 (3^2)
   + 13 (3, 2, 1) + 10 (3, 1^3) + 8 (2^3)
   + 5 (2^2, 1^2) + 4 (2, 1^4) + (1^6)}\\
\psi_{(5,1)}&=
{\scriptstyle 10 (6) + 9 (5, 1) + 15 (4, 2) + 9 (4, 1^2) + 4 (3^2)
   + 13 (3, 2, 1) + 11 (3, 1^3) + 8 (2^3)
   + 5 (2^2, 1^2) + 3 (2, 1^4) + 2 (1^6)}\\ 
\psi_{(4,2)}&=
{\scriptstyle 9 (6) + 9 (5, 1) + 16 (4, 2) + 9 (4, 1^2) + 4 (3^2)
   + 12 (3, 2, 1) + 11 (3, 1^3) + 8 (2^3)
   + 6 (2^2, 1^2) + 3 (2, 1^4) + (1^6)}\\
\psi_{(4,1^2)}&=
{\scriptstyle   8 (6) + 10 (5, 1) + 15 (4, 2) + 9 (4, 1^2) + 5 (3^2)
     + 12 (3, 2, 1) + 11 (3, 1^3) + 9 (2^3)
     + 5 (2^2, 1^2) + 4 (2, 1^4)}\\ 
     \psi_{(3^2)}&=
{\scriptstyle 7 (6) + 9 (5, 1) + 16 (4, 2) + 9 (4, 1^2) + 6 (3^2)
   + 12 (3, 2, 1) + 11 (3, 1^3) + 8 (2^3)
   + 4 (2^2, 1^2) + 5 (2, 1^4) + (1^6)}\\ 
   \psi_{(3,2,1)}&=
{\scriptstyle   6 (6) + 10 (5, 1) + 16 (4, 2) + 8 (4, 1^2) + 4 (3^2)
     + 14 (3, 2, 1) + 10 (3, 1^3) + 6 (2^3)
     + 4 (2^2, 1^2) + 6 (2, 1^4)}\\   
     \psi_{(3,1^3)}&=
{\scriptstyle 5 (6) + 10 (5, 1) + 16 (4, 2) + 9 (4, 1^2) + 3 (3^2)
   + 14 (3, 2, 1) + 9 (3, 1^3) + 7 (2^3)
   + 4 (2^2, 1^2) + 6 (2, 1^4) + (1^6)}\\
   \psi_{(2^3)}&=
{\scriptstyle    4 (6) + 8 (5, 1) + 16 (4, 2) + 8 (4, 1^2) + 4 (3^2)
      + 16 (3, 2, 1) + 8 (3, 1^3) + 8 (2^3)
      + 4 (2^2, 1^2) + 4 (2, 1^4)}\\    
      \psi_{(2^2,1^2)}&=
{\scriptstyle 3 (6) + 9 (5, 1) + 13 (4, 2) + 10 (4, 1^2) + 7 (3^2)
   + 16 (3, 2, 1) + 6 (3, 1^3) + 5 (2^3)
   + 7 (2^2, 1^2) + 3 (2, 1^4) + (1^6) }\\     
     \psi_{(2,1^4)}&=
{\scriptstyle   2 (6) + 8 (5, 1) + 12 (4, 2) + 12 (4, 1^2) + 6 (3^2)
     + 16 (3, 2, 1) + 8 (3, 1^3) + 4 (2^3)
     + 6 (2^2, 1^2) + 2 (2, 1^4) }\\
     \psi_{(1^6)}&={\scriptstyle (6)+5 (5, 1)+9 (4, 2)+10 (4, 1^2 )+5 (3^2)+16 (3, 2, 1)+10 (3, 1^3)+5 (2^3)+9 (2^2, 1^2)+5 (2, 1^4)+(1^6) }                           
\end{align*}

\vskip .2in

\noindent{\bf Acknowledgment:} The author is grateful to Erik Altman for invaluable help in computing the values of $f(n)$ in Table 2, particularly for the calculation of the entry for $n=10.$  The author also thanks the referee for a careful reading of the paper.

%We also have 
%\begin{proposition}$ Conj=(1-Lie)\left( \sum_{n\geq 1} Hk_n[Lie]\right).$
%\end{proposition}

% If using bibtex with separate .bib file
%\bibliographystyle{amsplain.bst}
%\bibliography{rps70.bib}

\bibliographystyle{amsplain.bst}

\end{document}